\documentclass[a4paper,11pt]{amsart}
\usepackage[utf8]{inputenc}

\usepackage{amssymb,amsfonts,amsmath}
\usepackage[english]{babel}
\usepackage{a4wide}
\usepackage{mathtools}
\usepackage{amsthm}
\usepackage{array,booktabs,tabularx}
\usepackage{placeins}
\usepackage{xcolor}
\usepackage[hidelinks]{hyperref}
\hypersetup{
  pdftitle={Separation properties of scrambled digital nets and related random point sets},
  pdfauthor={Kosuke Suzuki},
  pdfsubject={Local geometry of randomized quasi-Monte Carlo point sets},
  pdfkeywords={quasi-Monte Carlo, digital nets, Owen scrambling, matrix scrambling, linear scrambling, separation radius, mesh ratio, jittered sampling, Latin hypercube sampling}
}

\usepackage{url}
\newcommand{\PP}{\mathbb P}
\newcommand{\EE}{\mathbb E}
\newenvironment{romanenumerate}
  {\begin{enumerate}}
  {\end{enumerate}}

\theoremstyle{plain}
\newtheorem{theorem}{Theorem}[section]
\newtheorem{corollary}[theorem]{Corollary}
\newtheorem{lemma}[theorem]{Lemma}

\theoremstyle{definition}
\newtheorem{definition}[theorem]{Definition}
\newtheorem{remark}[theorem]{Remark}

\begin{document}

\title[Separation properties of scrambled point sets]{Separation properties of scrambled digital nets and related random point sets}
\author[K.~Suzuki]{Kosuke Suzuki}
\address[K.~Suzuki]{Faculty of Science, Yamagata University, 1-4-12 Kojirakawa-machi, Yamagata, 990\mbox{-}8560, Japan}
\email[]{kosuke-suzuki@sci.kj.yamagata-u.ac.jp}
\thanks{The work of K.S.\ is supported by JSPS KAKENHI Grant Numbers 24K06857 and 26K00620.}

\date{July 30, 2026}


\begin{abstract}
We study how standard randomization procedures affect the local geometry of
quasi-Monte Carlo point sets, as measured by their minimum distance and mesh
ratio.  Although probabilistic selection within structured lattice families
can produce quasi-uniform point sets, randomizing an existing low-discrepancy
construction need not preserve quasi-uniformity.  We first determine sharp
probabilistic orders for Monte Carlo, jittered, and Latin hypercube sampling,
whose mesh ratios diverge as positive powers of $N$.  The orders are
$\Theta_{\PP}(N^{1/d}(\log N)^{1/d})$ for Monte Carlo sampling,
$\Theta_{\PP}(N^{1/(d+1)})$ for jittered sampling, and, for $d\ge2$,
$\Theta_{\PP}(N^{1/d}(\log N)^{1/d})$ for Latin hypercube sampling.  We also
obtain a Weibull limit law for
the minimum distance of jittered samples.  For full Owen scrambling, every
family of fixed-$t$ nets has minimum distance
$O_{\PP}(N^{-3/(2d)})$, and its mesh ratio is therefore
$\Omega_{\PP}(N^{1/(2d)})$.  Under uniform coincidence and common-prefix conditions, these bounds are
sharp up to logarithmic factors.  Moreover, a
single full Owen scrambling of any $(t,d)$-sequence is almost surely
non-quasi-uniform.  By contrast, for matrix and linear scrambling of binary digital nets with fixed $t$
in dimension $d\ge2$, the mesh ratio is $O_{\PP}(\log N)$, whereas it is
$\Theta_{\PP}(\log N)$ in the separate balanced-prefix affine-tail model.  The model also yields the exact probabilistic
order for one-dimensional binary digital $(0,m,1)$-nets under matrix or linear
scrambling.  These results demonstrate that the geometric effect of
randomization is governed by whether it introduces local independence or
shared algebraic randomness.
\end{abstract}

\maketitle

\section{Introduction}

Quasi-Monte Carlo point sets are designed primarily for uniform distribution
and small integration error.  Their usefulness in approximation and
space-filling design depends on a different aspect of their geometry.  The
covering radius measures whether the whole domain is reached, whereas the
separation radius detects clusters of nearly coincident points.  A family is
quasi-uniform when the ratio of these two radii remains bounded.  In that
case, both radii have the optimal order $N^{-1/d}$.  This condition is central
in scattered data approximation
\cite{Wendland2005,SchabackWendland2006} and in the design of computer
experiments~\cite{PronzatoMuller2012,PronzatoZhigljavsky2023}.  Its role for
quasi-Monte Carlo constructions leads to a basic question: are QMC point sets
and sequences quasi-uniform?  Goda initiated the recent study of this question
for classical QMC constructions by proving that the two-dimensional Sobol'
sequence is not quasi-uniform~\cite{Goda2024}.  The question has since been
studied systematically for the two principal structured QMC families,
lattice-type constructions and digital nets and sequences
\cite{DGLPS25,DGS25}.  Goda, Hofer, and Suzuki have also shown that Halton
sequences are not quasi-uniform in any dimension $d\ge2$, together with
several further negative results for Halton-type sequences
\cite{GodaHoferSuzuki2026}.

The answer is largely positive on the lattice side.  Rank-1 and Fibonacci
lattice point sets, Frolov point sets, and Kronecker sequences have all been
analyzed from this perspective, and several of these classical constructions
combine low discrepancy with quasi-uniformity~\cite{DGLPS25}.  Moreover, for
every prime $N$ and fixed dimension, a positive proportion of the generating
vectors of rank-1 lattice point sets yield both bounded mesh ratio and
star-discrepancy of order
$N^{-1}(\log N)^d$~\cite[Theorem~3.7]{DGLPS25}.  Thus there are many lattice
point sets that are simultaneously suitable for integration and
space-filling.  Recent work has also developed constructive space-filling
lattice designs with provable quasi-uniformity~\cite{SakaiGoda2026}.

The picture for digital nets is markedly different.  Their minimum-distance
properties were previously investigated by seeking $(t,m,s)$-nets with
maximized minimum distance
\cite{GruenschlossHanikaSchwedeKeller2008,GruenschlossKeller2009}.  However,
the only explicit family currently known to combine low discrepancy and
quasi-uniformity in dimension $d\ge2$ is the two-dimensional
Larcher--Pillichshammer digital net~\cite{DGS25}.  No such family is known in
dimension $d\ge3$.  By contrast, several standard low-discrepancy digital
constructions have been proved not to be quasi-uniform
\cite{Goda2024,DGS25,Suzuki2025}.  This asymmetry raises the question of
whether the standard randomizations of digital nets improve their local
geometry.

These randomizations scramble the digits of a given deterministic net.  Full
Owen scrambling applies independent random permutations throughout the rooted
digit tree~\cite{Owe95}.  Matrix-based scrambling instead acts through random
nonsingular lower triangular
matrices~\cite{Matousek1998,Owen2003,DP10}.  Both preserve the net property and
are fundamental tools for randomized quasi-Monte Carlo integration.  The
matrix and linear scramblings considered here both use such matrices.  The
former uses a fixed digital shift and the latter a random one.  The
dependence properties of scrambled nets have been studied in detail
in~\cite{WiartLemieuxDong2021}, but their local space-filling geometry is not
determined by the net property alone.

Our main results for digital scrambling reveal a sharp contrast between the
two scrambling mechanisms.  Under full Owen scrambling, conditional on the
scrambled prefixes, the random subtrees below distinct prefixes are independent.
We show that this local independence creates pairs at a polynomially smaller
scale than $N^{-1/d}$ and hence produces polynomial growth of the mesh ratio.
By contrast, matrix and linear scrambling preserve a shared digital linear
structure.  We prove that, for binary digital fixed-$t$ nets in dimension
$d\ge2$ satisfying the coincidence bound, the mesh ratio is
$O_{\PP}(\log N)$.  Although this does not establish quasi-uniformity, it rules
out the polynomial deterioration seen under full Owen scrambling.  We further
show that the logarithmic order is attained in the balanced-prefix affine-tail
model and for one-dimensional binary digital nets.

For comparison, we also analyze Monte Carlo, jittered, and Latin hypercube
sampling as benchmarks for independent sampling, local stratification, and
marginal stratification, respectively.  We determine their sharp probabilistic
orders and show that all three have mesh ratios that diverge polynomially in
$N$.

\subsection{Summary of results}

The precise results are summarized in Table~\ref{tab:summary-results}.  The
dimension $d$ is fixed throughout.  In the digital-net results, the base and the quality
parameter $t$, when applicable, are also fixed, and $N=b^m\to\infty$.  The
balanced-prefix affine-tail row additionally assumes $m=kd$.

\begin{table}[htbp]
\caption{Summary of separation-radius and mesh-ratio bounds.}
\label{tab:summary-results}
\footnotesize
\setlength{\tabcolsep}{4pt}
\renewcommand{\arraystretch}{1.20}
\begin{tabularx}{\textwidth}{@{}>{\centering\arraybackslash}p{0.08\textwidth}>{\raggedright\arraybackslash}p{0.30\textwidth}>{\centering\arraybackslash}p{0.30\textwidth}>{\centering\arraybackslash}X@{}}
\toprule
\multicolumn{1}{c}{Section}
& \multicolumn{1}{l}{Model}
& Separation radius
& Mesh ratio \\
\midrule
\ref{subsec:monte-carlo}
& Monte Carlo
& $\Theta_{\PP}(N^{-2/d})$
& $\Theta_{\PP}\!\left(N^{1/d}(\log N)^{1/d}\right)$ \\
\addlinespace
\ref{subsec:jittered}
& Jittered sampling
& $\Theta_{\PP}(N^{-1/d-1/(d+1)})$
& $\Theta_{\PP}(N^{1/(d+1)})$ \\
\addlinespace
\ref{subsec:latin-hypercube}
& Latin hypercube sampling ($d\ge2$)
& $\Theta_{\PP}(N^{-2/d})$
& $\Theta_{\PP}\!\left(N^{1/d}(\log N)^{1/d}\right)$ \\
\midrule
\ref{sec:owen}
& Full Owen scrambling of fixed-$t$ nets with coincidence and common-prefix conditions
& $\Omega_{\PP}\!\left(N^{-3/(2d)}(\log N)^{-(d-1)/(2d)}\right)$,
  $O_{\PP}\!\left(N^{-3/(2d)}\right)$
& $\Omega_{\PP}\!\left(N^{1/(2d)}\right)$,
  $O_{\PP}\!\left(N^{1/(2d)}(\log N)^{(d-1)/(2d)}\right)$ \\
\midrule
\ref{subsec:higher-dimensional-matrix-linear}
& Matrix or linear scrambling of binary digital fixed-$t$ nets with coincidence bound ($d\ge2$)
& $\Omega_{\PP}(N^{-1/d}(\log N)^{-1})$
& $O_{\PP}(\log N)$ \\
\addlinespace
\ref{subsec:balanced-affine-tails}
& Balanced-prefix affine-tail model ($m=kd$)
& $\Theta_{\PP}(N^{-1/d}(\log N)^{-1})$
& $\Theta_{\PP}(\log N)$ \\
\addlinespace
\ref{subsec:one-dimensional-matrix-linear}
& Matrix or linear scrambling of binary digital $(0,m,1)$-nets
& $\Theta_{\PP}(N^{-1}(\log N)^{-1})$
& $\Theta_{\PP}(\log N)$ \\
\bottomrule
\end{tabularx}
\end{table}
\FloatBarrier

Together with the universal geometric bounds recalled in
Section~\ref{sec:prelim}, all rows determine the polynomial exponent of the
separation radius and mesh ratio.  The logarithmic factor is exact where a
$\Theta_{\PP}$ statement is given.  The Owen row retains a possible
logarithmic gap, while the higher-dimensional matrix and linear scrambling row
gives only one-sided logarithmic bounds.  The one-dimensional scrambling row
follows from the balanced-prefix affine-tail theorem.  In dimension one, full
Owen scrambling is included in the jittered-sampling row.

For fixed-shift matrix scrambling, the stated bounds are uniform in the shift.
The Owen bounds require a uniform elementary-interval coincidence bound and a
uniform cutoff on the total common-prefix depth.  For digital point sets, the
coincidence bound implies the cutoff.  Without these conditions, arbitrary
fixed-$t$ nets still satisfy the one-sided close-pair bound of
Corollary~\ref{cor:owen-upper-net}.

Several finer conclusions are not displayed in the table.  For jittered
sampling, we obtain an explicit Weibull limit law for the separation radius.
In one dimension, the first $b^m$ points of any $(0,1)$-sequence have the same
limit law after full Owen scrambling.  At the sequence level, a single full
Owen scrambling of any $(t,d)$-sequence almost surely destroys
quasi-uniformity.  Under the uniform conditions in the Owen row, we also
determine the almost-sure polynomial exponents of the separation radius and
mesh ratio.  The probabilistic and almost-sure exponent conclusions extend to
Halton sequences in pairwise coprime bases, with independent full Owen
scrambling applied in the respective bases.  Binary digital
$(0,1)$-sequences under matrix or linear scrambling are almost surely
non-quasi-uniform as well.

The remainder of the paper is organized as follows.
Section~\ref{sec:prelim} introduces the geometric quantities, probabilistic
notation, digital nets, and the scrambling models.
Section~\ref{sec:random-sampling} treats Monte Carlo, jittered, and Latin
hypercube sampling.  Section~\ref{sec:owen} studies full Owen scrambling using
prefix conditioning and second-moment arguments.  Section~\ref{sec:linear}
develops the affine-tail formulation for matrix and linear scrambling.

\FloatBarrier
\section{Preliminaries}\label{sec:prelim}

\subsection{Geometric quantities and quasi-uniformity}

We treat finite point sets as multisets throughout.  Cardinalities and
occupancies count multiplicity, and pairwise minima are over different indices.
Minimum-distance quantities are understood only for $N\ge2$.
We index an $N$-point set as $P=(\boldsymbol{x}_n)$, $0\le n<N$.  Following
the notation of~\cite{DGLPS25,DGS25}, define its $\ell^\infty$ covering radius, separation
radius, and mesh ratio by
\[
 h_\infty(P)
 :=
 \sup_{\boldsymbol{z}\in[0,1]^d}
 \min_{0\le n<N}
 \lVert\boldsymbol{z}-\boldsymbol{x}_n\rVert_\infty,
\]
\[
 q_\infty(P)
 :=
 \frac12
 \min_{0\le n<n'<N}
 \lVert\boldsymbol{x}_n-\boldsymbol{x}_{n'}\rVert_\infty,
\]
and
\[
 \rho_\infty(P):=\frac{h_\infty(P)}{q_\infty(P)},
\]
with the convention $\rho_\infty(P)=\infty$ when $q_\infty(P)=0$.  We also write
\[
 R_\infty(P):=2q_\infty(P)
 =\min_{0\le n<n'<N}
  \lVert\boldsymbol{x}_n-\boldsymbol{x}_{n'}\rVert_\infty
\]
for the minimum interpoint distance.
For an event $E$, we write $\chi(E)$ for its indicator.

By~\cite[Lemma~2.1]{PronzatoZhigljavsky2023}, every $N$-point set
$P$ in $[0,1]^d$, with $N\ge2$, satisfies
\begin{equation}\label{eq:universal-geometric-bounds}
 h_\infty(P)\ge \frac{1}{2N^{1/d}},
 \qquad
 q_\infty(P)\le C_d N^{-1/d},
\end{equation}
where $C_d>0$ depends only on $d$.  Consequently,
$\rho_\infty(P)\ge 1/(2C_d)$.  Thus the mesh ratio is always bounded below by a positive constant, and a
uniform upper bound forces both radii to have the optimal scale $N^{-1/d}$.
This motivates the following definition.

\begin{definition}
Let $\mathcal X$ be a family of finite point sets in $[0,1]^d$ with
$\sup_{P\in\mathcal X}|P|=\infty$.  It is called quasi-uniform if
\[
 \sup_{P\in\mathcal X}\rho_\infty(P)<\infty.
\]
For an infinite sequence $\mathcal S=(\boldsymbol{x}_n)_{n\ge0}$ in
$[0,1]^d$, write $P_N(\mathcal S)=(\boldsymbol{x}_n)_{0\le n<N}$.
The sequence $\mathcal S$ is called quasi-uniform if
\[
 \sup_{N\ge2}\rho_\infty(P_N(\mathcal S))<\infty.
\]
We will often study selected initial point sets, such as the subfamily
$(P_{b^m}(\mathcal S))_{m\ge1}$ consisting of the first $b^m$ points.
\end{definition}

By~\eqref{eq:universal-geometric-bounds}, a family with unbounded cardinalities
is quasi-uniform if and only if, uniformly over the family,
\[
 h_\infty(P)\le C N^{-1/d},
 \qquad
 q_\infty(P)\ge c N^{-1/d}
\]
for some constants $c,C>0$.  In particular, the first bound in
\eqref{eq:universal-geometric-bounds} converts an upper bound for the
separation radius into a lower bound for the mesh ratio.

Since all norms on $\mathbb R^d$ are equivalent, boundedness of the mesh ratio
does not depend on the chosen $\ell^p$ norm when $d$ is fixed.  The
$\ell^\infty$ norm is used throughout because it is adapted to digital
partitions.

The following quantitative form of the monotonicity argument in
\cite[proof of Lemma~1.3]{DGLPS25} will be used repeatedly.

\begin{lemma}\label{lem:initial-segment-comparison}
Let $\mathcal S=(\boldsymbol{x}_n)_{n\ge0}$ be a sequence in $[0,1]^d$, and
write $P_N=P_N(\mathcal S)$.  Let $2\le M_1<M_2<\cdots$ be integers.
If $M_r\le N<M_{r+1}$, then
\[
 R_\infty(P_{M_{r+1}})
 \le R_\infty(P_N)
 \le R_\infty(P_{M_r}),
 \qquad
 h_\infty(P_{M_{r+1}})
 \le h_\infty(P_N)
 \le h_\infty(P_{M_r}).
\]
Consequently,
\[
 \frac{2h_\infty(P_{M_{r+1}})}
      {R_\infty(P_{M_r})}
 \le
 \rho_\infty(P_N)
 \le
 \frac{2h_\infty(P_{M_r})}
      {R_\infty(P_{M_{r+1}})},
\]
with the extended-real convention if a denominator vanishes.
\end{lemma}

\subsection{Order notation}\label{subsec:prob-order}

For deterministic quantities, a subscript on an order symbol indicates the
parameters on which its implicit constants may depend.  Thus, for example,
$f_N=O_{d,t}(g_N)$ means that $|f_N|\le C_{d,t}g_N$ for some constant
$C_{d,t}$ depending only on $d$ and $t$.  We use the same convention for
$\Omega_{d,t}$, $\Theta_{d,t}$, and $\asymp_{d,t}$, with the list of
parameters changed as appropriate.

Let $(X_m)_{m\ge1}$ be random variables and let $(a_m)_{m\ge1}$ be positive
deterministic numbers.  We write $X_m=O_{\PP}(a_m)$ if, for every
$\varepsilon>0$, there exist constants $C_\varepsilon>0$ and
$m_\varepsilon$ such that
\[
 \PP\bigl(|X_m|>C_\varepsilon a_m\bigr)\le\varepsilon
 \qquad\text{for all }m\ge m_\varepsilon.
\]
For positive random variables, we write $X_m=\Omega_{\PP}(a_m)$ if
$a_m/X_m=O_{\PP}(1)$, and
$X_m=\Theta_{\PP}(a_m)$ if both $X_m=O_{\PP}(a_m)$ and
$X_m=\Omega_{\PP}(a_m)$.  Equivalently, $X_m/a_m$ is bounded and bounded
away from zero in probability.

\begin{lemma}\label{lem:probabilistic-tools}
Let $Y_m,Z_m$ be nonnegative random variables, and let $a_m>0$ be deterministic.
\begin{romanenumerate}
\item If $\EE[Y_m]=O(a_m)$, then $Y_m=O_{\PP}(a_m)$.  If $Z_m$ is integer-valued and
$\EE[Z_m]\to0$, then $\PP(Z_m>0)\to0$.
\item If $\EE[Z_m]>0$, then
\[
 \PP(Z_m=0)
 \le
 \frac{\operatorname{Var}(Z_m)}{(\EE[Z_m])^2}.
\]
\end{romanenumerate}
\end{lemma}

\begin{proof}
For every $u>0$, Markov's inequality gives
$\PP(Y_m>u a_m)\le \EE[Y_m]/(u a_m)$.  This proves the first assertion in part~(i), and the second follows from
$\PP(Z_m>0)\le \EE[Z_m]$ when $Z_m$ is nonnegative and integer-valued.
For part~(ii), $\{Z_m=0\}\subseteq
\{|Z_m-\EE[Z_m]|\ge \EE[Z_m]\}$, so the estimate follows from
Chebyshev's inequality.
\end{proof}

\begin{lemma}[First Borel--Cantelli lemma]
\label{lem:first-borel-cantelli}
If $(A_m)_{m\ge1}$ is a sequence of events such that
\[
 \sum_{m=1}^{\infty}\PP(A_m)<\infty,
\]
then almost surely only finitely many of the events $A_m$ occur.
\end{lemma}

\subsection{Digital nets and sequences}

Let $b\ge2$ be an integer.

\begin{definition}
\label{def:level-elementary-interval}
For $\boldsymbol{k}=(k_1,\ldots,k_d)\in\mathbb N_0^d$ and
$\boldsymbol{a}=(a_1,\ldots,a_d)$ with $0\le a_j<b^{k_j}$, put
\[
 E(\boldsymbol{k},\boldsymbol{a})
 :=
 \prod_{j=1}^d
 \left[\frac{a_j}{b^{k_j}},\frac{a_j+1}{b^{k_j}}\right).
\]
This is called a level-$\boldsymbol{k}$ elementary interval in base $b$, and
$\mathcal E_{\boldsymbol{k}}$ denotes the collection of all such intervals.
Its volume is $b^{-|\boldsymbol{k}|}$, where
$|\boldsymbol{k}|=k_1+\cdots+k_d$.  A $b$-adic elementary interval is a
level-$\boldsymbol{k}$ elementary interval for some $\boldsymbol{k}$.
\end{definition}

\begin{definition}
Let $m,t\in\mathbb N_0$ with $t\le m$.  A $b^m$-point set $P$ in
$[0,1)^d$ is called a $(t,m,d)$-net in base $b$ if every $b$-adic elementary
interval of volume $b^{t-m}$ contains exactly $b^t$ points of $P$.
\end{definition}

\begin{definition}
An infinite sequence $\mathcal S=(\boldsymbol{x}_n)_{n\ge0}$ in $[0,1)^d$ is
called a $(t,d)$-sequence in base $b$ if, for every $k\ge0$ and every
$m\ge t$, the point set
$\{\boldsymbol{x}_n:kb^m\le n<(k+1)b^m\}$ is a $(t,m,d)$-net in base $b$.
\end{definition}

\begin{definition}[Tezuka's $(0,\boldsymbol{e},d)$-sequence~\cite{Te13}]
\label{def:zero-e-sequence}
Let $\boldsymbol{e}=(e_1,\ldots,e_d)\in\mathbb N^d$.  An infinite sequence
$\mathcal S=(\boldsymbol{x}_n)_{n\ge0}$ in $[0,1)^d$ is called a
$(0,\boldsymbol{e},d)$-sequence in base $b$ if, for every
$\boldsymbol{c}=(c_1,\ldots,c_d)\in\mathbb N_0^d$ and every
$r\in\mathbb N_0$, with
\[
 M=\sum_{j=1}^d e_jc_j,
 \qquad
 \boldsymbol{q}=(e_1c_1,\ldots,e_dc_d),
\]
the point set $\{\boldsymbol{x}_n:rb^M\le n<(r+1)b^M\}$ contains exactly one
point in each level-$\boldsymbol{q}$ elementary interval.
\end{definition}

By~\cite[Theorem~1]{Te13}, a generalized Niederreiter sequence constructed
from pairwise coprime polynomials $p_1,\ldots,p_d$ is a
$(0,\boldsymbol{e},d)$-sequence with $e_j=\deg p_j$.  Sobol' sequences and
original Niederreiter sequences are special cases.
See~\cite[Section~8.1]{DP10}.

It follows from Niederreiter's dispersion
bound~\cite[Theorem~6.10]{Nie92} that every $(t,m,d)$-net $P$ in base $b$, with
$N=b^m$, satisfies
\begin{equation}\label{eq:net-covering}
 h_\infty(P)
 \le b^{-\lfloor(m-t)/d\rfloor}
 \le b^{(d-1+t)/d}N^{-1/d}.
\end{equation}
Here Niederreiter's dispersion with respect to the maximum metric is precisely
the covering radius $h_\infty$ used in this paper.

Suppose now that $b$ is a prime power.  Fix a bijection
$\varphi:\mathbb F_b\to\{0,\ldots,b-1\}$ with $\varphi(0)=0$, and use it to
identify field elements with base-$b$ digits.  Let $n\in\mathbb N\cup\{\infty\}$, with $n\ge m$ when $n<\infty$,
and put $I_n=\{1,\ldots,n\}$ for finite $n$ and $I_\infty=\mathbb N$.
Given generating matrices $C_1,\ldots,C_d\in\mathbb F_b^{n\times m}$, where $\mathbb F_b^{\infty\times m}$ means
$\mathbb F_b^{\mathbb N\times m}$, proceed as follows.  For an integer $0\le a<b^m$, write
$a=a_0+a_1b+\cdots+a_{m-1}b^{m-1}$ and regard
$\boldsymbol{a}=(\varphi^{-1}(a_0),\ldots,\varphi^{-1}(a_{m-1}))^\top$
as an element of $\mathbb F_b^m$.  We usually index the resulting point by
the integer $a$.  In Section~\ref{sec:linear}, where the linear structure of
the input labels is used, we instead index it by the vector
$\boldsymbol{a}$ itself.  Write
\[
 C_j\boldsymbol{a}=(y_{a,j,r})_{r\in I_n}
\]
and define
\[
 x_{a,j}=\sum_{r\in I_n} \varphi(y_{a,j,r})b^{-r}.
\]  The resulting
$b^m$ points form the digital net generated by $C_1,\ldots,C_d$.  For finite
$n$, zero digits are appended below level $n$.  Thus $m$ is the number of input
digits, whereas $n$ is the output precision and need not equal $m$.

A digital sequence is generated by matrices
$C_1,\ldots,C_d\in\mathbb F_b^{\mathbb N\times\mathbb N}$.
The first $b^m$ points are obtained from the first $m$ columns, hence form a
digital net with generating matrices in $\mathbb F_b^{\infty\times m}$.
Whenever prefixes of digital points are discussed below, they refer to these
output-digit expansions.

\subsection{Randomizations}\label{subsec:randomizations}

We first define full Owen scrambling~\cite{Owe95}, also called Owen's nested
uniform scrambling.  It is defined for every integer base $b\ge2$.  Write
$\mathcal D_b=\{0,\ldots,b-1\}$.  Independently for each coordinate $j$ and
each finite input prefix $a_1\cdots a_{r-1}\in\mathcal D_b^{r-1}$, choose a
uniform random permutation $\pi_{j,a_1\cdots a_{r-1}}$ of $\mathcal D_b$, with
all these permutations mutually independent, and replace the input digits
$(a_r)_{r\ge1}$ by
\[
 \widetilde a_r
 =\pi_{j,a_1\cdots a_{r-1}}(a_r),
 \qquad r\ge1.
\]
Equivalently, it is a random automorphism of the rooted $b$-ary digit tree.

In the prime-power digital setting, let $\oplus$ denote digitwise addition
over $\mathbb F_b$, and write $\boldsymbol{y}_{n,j}\in\mathbb F_b^{\mathbb N}$
for the output digit vector of the $j$-th coordinate, using zero padding as
above.  For matrix-based scrambling, see
\cite{Matousek1998,Owen2003,DP10}.  For each coordinate $j$, let $L_j$ be an
independent random
infinite nonsingular lower triangular matrix whose entries below the diagonal
are independent and uniform over $\mathbb F_b$ and whose diagonal entries are
independent and uniform over $\mathbb F_b\setminus\{0\}$.  Given a deterministic
digital shift
$\boldsymbol{\delta}=(\boldsymbol{\delta}_1,\ldots,\boldsymbol{\delta}_d)$, the
matrix scrambling is defined by
\[
 \boldsymbol{y}_{n,j}
 \longmapsto
 L_j\boldsymbol{y}_{n,j}\oplus\boldsymbol{\delta}_j.
\]
For a digital point set $P_{b^m}$, write
$P_{b^m}^{\mathrm{mat},\boldsymbol{\delta}}$ for the resulting point set.
Probability statements for this model are with respect to the matrices.
An exceptional fixed shift may produce the endpoint $1$.  When the net
property is used for covering, the corresponding elementary intervals are
then understood through their closures.  This leaves all covering estimates
unchanged.
If $\boldsymbol{\delta}$ is
instead random, with all its digits independent and uniform over $\mathbb F_b$
and independent of the matrices, the same transformation is called a linear
scramble.  Write $P_{b^m}^{\mathrm{lin}}$ for the resulting point set.
Probability is then taken jointly over the matrices and the shift.
We write $P_{b^m}^{\mathrm{scr}}$ for either
$P_{b^m}^{\mathrm{mat},\boldsymbol{\delta}}$ with fixed
$\boldsymbol{\delta}$ or $P_{b^m}^{\mathrm{lin}}$.  Thus each statement about
$P_{b^m}^{\mathrm{scr}}$ represents two separate statements, one under the
matrix-scrambling law and one under the linear-scrambling law.
For a digital sequence, the same matrices, and the same shifts in the linear
case, are used at every level.  We use the analogous notation
$\mathcal S^{\mathrm{scr}}$ and $P_N^{\mathrm{scr}}$ for the scrambled
sequence and its initial point sets.

\section{Monte Carlo sampling, jittered sampling, and Latin hypercube sampling}\label{sec:random-sampling}

This section provides benchmark geometric scales for three standard random
sampling schemes.  We begin with independent Monte Carlo sampling, then turn
to jittered sampling, whose local cell structure anticipates the digital
arguments below, and finally consider random Latin hypercube sampling.  The
jittered model also connects this section to full Owen scrambling.  In
dimension one, the first $b^m$ points of any $(0,1)$-sequence after full Owen
scrambling form a jittered sample.  Thus the one-dimensional jittered result
applies directly to the model studied in Section~\ref{sec:owen}.

\begin{definition}
For $\alpha,\lambda>0$, let $W_{\alpha,\lambda}$ denote a Weibull random
variable with
\[
 \PP(W_{\alpha,\lambda}>t)=\exp(-\lambda t^\alpha),
 \qquad t\ge0.
\]
Here $\alpha$ is the shape parameter and $\lambda^{-1/\alpha}$ is the
scale parameter.
\end{definition}

\subsection{Monte Carlo sampling}\label{subsec:monte-carlo}

Let $\boldsymbol{X}_0,\dots,\boldsymbol{X}_{N-1}$ be independent uniformly distributed random points in
$[0,1]^d$, and put $P_N^{\mathrm{MC}}=(\boldsymbol{X}_i)_{0\le i<N}$ and
$R_N^{\mathrm{MC}}=R_\infty(P_N^{\mathrm{MC}})$.
The limit law for the minimum distance $R_N^{\mathrm{MC}}$ follows from the
Poisson approximation theorem of Schulte and
Th\"ale~\cite[Corollary~2]{SchulteThale2016}, applied with the $\ell^\infty$
distance, while the covering-radius estimate follows from Reznikov and
Saff~\cite[Corollary~2.3]{ReznikovSaff2016}, applied to the unit cube with the
Euclidean metric, together with equivalence of the Euclidean and
$\ell^\infty$ norms.  In particular, Monte Carlo sampling has a much smaller
minimum distance than the optimal scale $N^{-1/d}$.

\begin{theorem}
\label{thm:mc-geometry}
Let $d\ge1$ be fixed.  Then, as $N\to\infty$,
\[
 N^{2/d}R_N^{\mathrm{MC}}
 \xrightarrow{\mathrm d}
 W_{d,2^{d-1}},
 \qquad
 h_\infty(P_N^{\mathrm{MC}})
 =\Theta_{\PP}\left(\left(\frac{\log N}{N}\right)^{1/d}\right).
\]
Consequently,
\[
 R_N^{\mathrm{MC}}=\Theta_{\PP}(N^{-2/d}),
 \qquad
 \rho_\infty(P_N^{\mathrm{MC}})
 =\Theta_{\PP}\left(N^{1/d}(\log N)^{1/d}\right).
\]
\end{theorem}

\subsection{Jittered sampling}\label{subsec:jittered}

We next consider jittered sampling.  Its discrepancy properties have been
studied in~\cite{PausingerSteinerberger2016,Doerr2022}.  Here we consider instead
its closest-pair geometry.  Let $n\ge1$, set $N=n^d$, and divide
$[0,1]^d$ into the $N$ cubes
\[
Q_{\boldsymbol{i}}
=
\prod_{j=1}^d
\left[\frac{i_j}{n},\frac{i_j+1}{n}\right),
\qquad
 \boldsymbol{i}=(i_1,\dots,i_d)\in\{0,\dots,n-1\}^d.
\]
In each cube $Q_{\boldsymbol{i}}$ place one point $\boldsymbol{X}_{\boldsymbol{i}}$, independently and uniformly in
that cube.  Put
$P_N^{\mathrm{jit}}=(\boldsymbol{X}_{\boldsymbol{i}})_{\boldsymbol{i}\in\{0,\dots,n-1\}^d}$
and $R_N^{\mathrm{jit}}=R_\infty(P_N^{\mathrm{jit}})$.

\begin{theorem}\label{thm:jittered-minimum-distance}
Let $d\ge1$ be fixed and let $N=n^d$.  Then, as $n\to\infty$,
\[
 nN^{1/(d+1)}R_N^{\mathrm{jit}}
 \xrightarrow{\mathrm d}
 W_{d+1,d\,2^{d-2}}.
\]
Moreover, there is a constant $C_d>0$ such that, for every $n\ge2$ and
$0<\delta<1/2$,
\begin{equation}\label{eq:jittered-elementary-tails}
 \PP\left(R_N^{\mathrm{jit}}\le\frac{\delta}{n}\right)
 \le C_dN\delta^{d+1},
 \qquad
 \PP\left(R_N^{\mathrm{jit}}>\frac{\delta}{n}\right)
 \le \exp\left(-\frac16N\delta^{d+1}\right).
\end{equation}
In particular,
\[
 R_N^{\mathrm{jit}}
 =\Theta_{\PP}\left(n^{-1}N^{-1/(d+1)}\right)
 =\Theta_{\PP}\left(N^{-1/d-1/(d+1)}\right).
\]
\end{theorem}

\begin{proof}
The case $t=0$ is immediate.  Fix $t>0$ and set
$r_N:=t/(nN^{1/(d+1)})$ and $\delta_N:=nr_N=tN^{-1/(d+1)}$.
By the definition of the Weibull distribution, the desired convergence is
obtained by evaluating $\PP(R_N^{\mathrm{jit}}>r_N)$ and proving that
\begin{equation}\label{eq:jittered-survival-target}
 \PP(R_N^{\mathrm{jit}}>r_N)
 \longrightarrow
 \exp\bigl(-d2^{d-2}t^{d+1}\bigr).
\end{equation}
We first establish the close-pair estimates for a general normalized threshold
$0<\delta<1/2$, corresponding to the distance threshold $\delta/n$.
Besides proving the finite-level bounds in
\eqref{eq:jittered-elementary-tails}, these estimates will later be applied
with $\delta=\delta_N$.

Let $\mathcal N_n$ be the set of unordered pairs of distinct grid cubes whose
closures intersect.  Equivalently,
\[
 \{Q_{\boldsymbol{i}},Q_{\boldsymbol{j}}\}\in\mathcal N_n
 \quad\Longleftrightarrow\quad
 \boldsymbol{i}-\boldsymbol{j}\in\{-1,0,1\}^d\setminus\{\boldsymbol{0}\}.
\]
For $e=\{Q_{\boldsymbol{i}},Q_{\boldsymbol{j}}\}\in\mathcal N_n$, let
$I_e(\delta)$ indicate that the two jittered points in these cubes are at
$\ell^\infty$-distance at most $\delta/n$, and put
\[
 W_N(\delta):=\sum_{e\in\mathcal N_n}I_e(\delta).
\]
Two cubes whose closures do not intersect are separated by at least $1/n$ in
some coordinate.  Hence
\begin{equation}\label{eq:jittered-zero-count}
 \left\{R_N^{\mathrm{jit}}>\frac{\delta}{n}\right\}
 =\{W_N(\delta)=0\}.
\end{equation}

Suppose that the indices of two neighboring cubes differ in exactly $c$
coordinates.  The number of such unordered pairs is
\[
 M_{n,c}
 =2^{c-1}\binom dc(n-1)^c n^{d-c}.
\]
Translate the two coordinate intervals and scale them by $n$.  In each of
the $c$ coordinates in which the cube indices differ, the intervals become
$[0,1)$ and $[1,2)$.  For independent uniform points $U$ and $V$ in these
intervals, the condition $|U-V|\le\delta$ cuts out a right triangle of area
$\delta^2/2$.  In each remaining coordinate, both points are uniform in the
same interval $[0,1)$, and the band $|U-V|\le\delta$ has area
$2\delta-\delta^2$.  Since the coordinates are independent, the
corresponding close-pair probability is
\begin{equation}\label{eq:jittered-local-probability}
 p_c(\delta)
 =\left(\frac{\delta^2}{2}\right)^c
  (2\delta-\delta^2)^{d-c},
\end{equation}
and therefore
\begin{equation}\label{eq:jittered-mean-general}
 \EE[W_N(\delta)]
 =\sum_{c=1}^d M_{n,c}p_c(\delta).
\end{equation}
Since $W_N(\delta)$ is nonnegative and integer-valued,
\eqref{eq:jittered-zero-count} and Markov's inequality give
\[
 \PP\left(R_N^{\mathrm{jit}}\le\frac{\delta}{n}\right)
 =\PP\bigl(W_N(\delta)\ge1\bigr)
 \le \EE[W_N(\delta)]
 \le C_dN\delta^{d+1},
\]
where the last inequality follows from
$p_c(\delta)\le C_d\delta^{d+c}$ and $M_{n,c}\le C_dN$.

For the opposite tail, for each
$(i_2,\dots,i_d)\in\{0,\dots,n-1\}^{d-1}$, pair consecutive cubes in the
first coordinate as
$\{Q_{(2k,i_2,\dots,i_d)},Q_{(2k+1,i_2,\dots,i_d)}\}$ for
$0\le k<\lfloor n/2\rfloor$.
This gives
$M_n=\lfloor n/2\rfloor n^{d-1}\ge N/3$ disjoint face-sharing pairs, whose
close-pair events are independent.  By
\eqref{eq:jittered-local-probability}, each occurs with probability
\[
 q_\delta
 =\frac{\delta^2}{2}(2\delta-\delta^2)^{d-1}
 \ge \frac12\delta^{d+1},
\]
since we have assumed $0<\delta<1/2$.
Therefore,
\[
 \PP\left(R_N^{\mathrm{jit}}>\frac{\delta}{n}\right)
 \le (1-q_\delta)^{M_n}
 \le \exp(-M_nq_\delta)
 \le \exp\left(-\frac16N\delta^{d+1}\right),
\]
which completes the proof of \eqref{eq:jittered-elementary-tails}.

We now apply these estimates at the normalized threshold $\delta_N$, which
corresponds to the critical radius $r_N=\delta_N/n$.  For all sufficiently
large $N$, $0<\delta_N<1/2$.  Hence, by \eqref{eq:jittered-zero-count}, it
remains to determine the limit of $\PP(W_N(\delta_N)=0)$.  Set
$\lambda_N:=\EE[W_N(\delta_N)]$.  The $c=1$ term in
\eqref{eq:jittered-mean-general} satisfies
\[
 M_{n,1}p_1(\delta_N)
 =d2^{d-2}N\delta_N^{d+1}(1+o(1))
 \longrightarrow d2^{d-2}t^{d+1},
\]
whereas, for every $c\ge2$,
\[
 M_{n,c}p_c(\delta_N)
 =O_{d,t}(N\delta_N^{d+c})
 =O_{d,t}\bigl(N^{(1-c)/(d+1)}\bigr)
 \longrightarrow0.
\]
Consequently,
\begin{equation}\label{eq:jittered-mean-limit}
 \lambda_N\longrightarrow \lambda:=d2^{d-2}t^{d+1}.
\end{equation}

It remains to show that $W_N(\delta_N)$ is asymptotically Poisson.  The indicators
$(I_e(\delta_N))_{e\in\mathcal N_n}$ have a dependency graph of degree bounded
only in terms of $d$.  Indeed, two indicators are independent whenever their
cube pairs are disjoint.  For
$e=\{Q_{\boldsymbol{i}},Q_{\boldsymbol{j}}\}$, let $B_e$ be the set of pairs in
$\mathcal N_n$ sharing a cube with $e$.  Thus $B_e$ consists of the pairs of the form
$\{Q_{\boldsymbol{i}},Q_{\boldsymbol{k}}\}$ or
$\{Q_{\boldsymbol{j}},Q_{\boldsymbol{k}}\}$ that belong to $\mathcal N_n$.
Define
\[
 b_1:=\sum_{e\in\mathcal N_n}\sum_{f\in B_e}
      \EE[I_e(\delta_N)]\EE[I_f(\delta_N)],
 \qquad
 b_2:=\sum_{e\in\mathcal N_n}
      \sum_{\substack{f\in B_e\\f\ne e}}
      \EE[I_e(\delta_N)I_f(\delta_N)].
\]
With the convention
$d_{\mathrm{TV}}(\mu,\nu)
:=\sup_{A\subset\mathbb Z_{\ge0}}|\mu(A)-\nu(A)|$, the total variation norm used
in~\cite[Theorem~1]{ArratiaGoldsteinGordon1989} is twice the distance above.
That theorem therefore gives
\begin{equation}\label{eq:jittered-chen-stein}
 d_{\mathrm{TV}}\bigl(\mathcal L(W_N(\delta_N)),\operatorname{Po}(\lambda_N)\bigr)
 \le b_1+b_2,
\end{equation}
because independence outside $B_e$ makes the third error term vanish.

Since each cube has at most $3^d-1$ neighboring cubes, we have
$|\mathcal N_n|\le 3^dN$ and $|B_e|\le 2(3^d-1)$ for every
$e\in\mathcal N_n$.
Moreover, by \eqref{eq:jittered-local-probability},
$\EE[I_e(\delta_N)]=O_{d,t}(N^{-1})$ uniformly in $e$, and hence
$b_1=O_{d,t}(N^{-1})$.

To estimate $b_2$, fix distinct $e,f\in\mathcal N_n$ with $f\in B_e$.
They share exactly one cube, so write
$e=\{Q_{\boldsymbol{i}},Q_{\boldsymbol{j}}\}$ and
$f=\{Q_{\boldsymbol{i}},Q_{\boldsymbol{k}}\}$.
Then
$\EE[I_e(\delta_N)I_f(\delta_N)]$
is the probability that both close-pair events occur.  Put
$\boldsymbol{u}=\boldsymbol{j}-\boldsymbol{i}$ and
$\boldsymbol{v}=\boldsymbol{k}-\boldsymbol{i}$, and define
\[
\mathcal F(\boldsymbol{u})
=\{(\ell,+)\mid u_\ell=1\}\cup\{(\ell,-)\mid u_\ell=-1\},
\]
with the analogous definition for $\mathcal F(\boldsymbol{v})$.  After
translating $Q_{\boldsymbol{i}}$ to $[0,1)^d$ and scaling by $n$, the point in
$Q_{\boldsymbol{i}}$ must lie in $[1-\delta_N,1)$ in coordinate $\ell$ when
$(\ell,+)$ belongs to either set, and in $[0,\delta_N)$ when $(\ell,-)$
belongs to either set.  The joint event is therefore impossible if the union
contains both $(\ell,+)$ and $(\ell,-)$ for some $\ell$, because
$\delta_N<1/2$.  Otherwise $\boldsymbol{u}$ and $\boldsymbol{v}$ are distinct
nonzero vectors, and hence
$|\mathcal F(\boldsymbol{u})\cup\mathcal F(\boldsymbol{v})|\ge2$, since equality with a singleton would force both vectors to be the same
signed coordinate vector.  The signed coordinates in the union then have
distinct coordinate indices.  Thus the admissible region for the point in
$Q_{\boldsymbol{i}}$ has volume at most $\delta_N^2$.  Once this point is fixed,
the admissible region for each of the points in $Q_{\boldsymbol{j}}$ and
$Q_{\boldsymbol{k}}$ has volume at most $(2\delta_N)^d$.  Since the three points
in the distinct jittered cubes are independent, we obtain
\[
 \EE[I_e(\delta_N)I_f(\delta_N)]
 \le \delta_N^2(2\delta_N)^{2d}
 \le C_d\delta_N^{2d+2}.
\]
Consequently,
\[
 \sup_{\substack{e\in\mathcal N_n,\ f\in B_e\\f\ne e}}
 \EE[I_e(\delta_N)I_f(\delta_N)]
 \le C_d\delta_N^{2d+2}=O_{d,t}(N^{-2}),
\]
and therefore
\[
 b_2
 \le |\mathcal N_n|\sup_{e\in\mathcal N_n}|B_e|
 \sup_{\substack{e\in\mathcal N_n,\ f\in B_e\\f\ne e}}
 \EE[I_e(\delta_N)I_f(\delta_N)]
 =O_{d,t}(N^{-1}).
\]
It follows from \eqref{eq:jittered-chen-stein} that
\[
 d_{\mathrm{TV}}\bigl(\mathcal L(W_N(\delta_N)),\operatorname{Po}(\lambda_N)\bigr)
 \longrightarrow0.
\]
Together with \eqref{eq:jittered-mean-limit}, this yields
\[
 \PP(R_N^{\mathrm{jit}}>r_N)
 =\PP(W_N(\delta_N)=0)
 =e^{-\lambda_N}+o(1)
 \longrightarrow
 \exp\bigl(-d2^{d-2}t^{d+1}\bigr),
\]
which proves \eqref{eq:jittered-survival-target} and the desired convergence.
\end{proof}

The covering lower bound in~\eqref{eq:universal-geometric-bounds}
and the fact that each grid cube contains one point give
$1/(2n)\le h_\infty(P_N^{\mathrm{jit}})\le 1/n$.
Combining this with Theorem~\ref{thm:jittered-minimum-distance} gives the following.

\begin{corollary}
For jittered sampling in $[0,1]^d$,
\[
 h_\infty(P_N^{\mathrm{jit}})\asymp_d N^{-1/d},
 \qquad
 \rho_\infty(P_N^{\mathrm{jit}})
 =\Theta_{\PP}\left(N^{1/(d+1)}\right).
\]
\end{corollary}

\subsection{Random Latin hypercube sampling}\label{subsec:latin-hypercube}

We finally consider standard random Latin hypercube sampling, introduced by
McKay, Beckman and Conover~\cite{McKayBeckmanConover1979}.  See also
Stein~\cite{Stein1987} for large-sample properties.  Let
$\pi_1,\dots,\pi_d$ be independent uniformly distributed permutations of
$\{0,\dots,N-1\}$, and let $U_{i,j}$ be independent uniform random variables
on $[0,1)$.  Define
\begin{equation}\label{eq:lhs-construction}
\boldsymbol{X}_i=\left(\frac{\pi_1(i)+U_{i,1}}{N},\dots,
          \frac{\pi_d(i)+U_{i,d}}{N}\right),
\qquad i=0,\dots,N-1,
\end{equation}
and put $P_N^{\mathrm{LHS}}=(\boldsymbol{X}_i)_{0\le i<N}$ and
$R_N^{\mathrm{LHS}}=R_\infty(P_N^{\mathrm{LHS}})$.
The Latin hypercube constraint gives exact one-dimensional stratification, but
it does not impose one point in each $d$-dimensional cube at the natural scale
$N^{-1/d}$.  As a consequence, for fixed dimension $d\ge2$ its mesh ratio has
the same order as independent Monte Carlo sampling.

\begin{lemma}
\label{lem:lhs-one-coordinate-close}
Let $N\ge2$, and let $Y=(A+U)/N$ and $Z=(B+V)/N$, where $(A,B)$ is a
uniformly distributed ordered pair of distinct elements of
$\{0,\dots,N-1\}$ and $U,V$ are independent uniform random variables on
$[0,1)$.  Put $p_N(r)=\PP(|Y-Z|\le r)$.  Then, for $0\le r\le1$,
\[
p_N(r)=
\begin{cases}
 Nr^2, & 0\le r\le N^{-1},\\[1mm]
 \displaystyle\frac{2r-r^2-N^{-1}}{1-N^{-1}},
   & N^{-1}\le r\le1.
\end{cases}
\]
In particular,
\begin{equation}\label{eq:lhs-one-coordinate-close}
 p_N(r)\le 2r
 \qquad (r\ge0).
\end{equation}
\end{lemma}

\begin{proof}
Let $\widetilde Y,\widetilde Z$ be independent uniform random variables on
$[0,1)$, and set
$\widetilde A=\lfloor N\widetilde Y\rfloor$ and
$\widetilde B=\lfloor N\widetilde Z\rfloor$.  Conditional on
$\widetilde A\ne\widetilde B$, the ordered pair
$(\widetilde A,\widetilde B)$ is uniform over distinct stratum labels, while
the positions within the two strata remain independent and uniform.  Hence
$(Y,Z)$ has the same law as $(\widetilde Y,\widetilde Z)$ under this
conditioning, and
\[
 p_N(r)
 =\PP\bigl(|\widetilde Y-\widetilde Z|\le r
      \mid \widetilde A\ne\widetilde B\bigr)
 =\frac{\PP\bigl(|\widetilde Y-\widetilde Z|\le r,
                  \ \widetilde A\ne\widetilde B\bigr)}{1-N^{-1}},
\]
since $\PP(\widetilde A\ne\widetilde B)=1-N^{-1}$.

If $0\le r\le N^{-1}$, closeness is possible only for adjacent strata.
Since $\PP(|A-B|=1)=2/N$ and, conditionally on $B=A+1$,
$|Y-Z|\le r$ is equivalent to $U-V\ge1-Nr$, we obtain
\[
 p_N(r)=\frac{2}{N}\frac{(Nr)^2}{2}=Nr^2.
\]
If $N^{-1}\le r\le1$, then
$\{\widetilde A=\widetilde B\}\subset
\{|\widetilde Y-\widetilde Z|\le r\}$, because a stratum has length
$N^{-1}$.  Therefore
\[
 p_N(r)=
 \frac{\PP(|\widetilde Y-\widetilde Z|\le r)-\PP(\widetilde A=\widetilde B)}{1-N^{-1}}
 =\frac{2r-r^2-N^{-1}}{1-N^{-1}}.
\]
The bound \eqref{eq:lhs-one-coordinate-close} follows immediately from
these formulas.  For $r>1$, it is trivial.
\end{proof}

\begin{theorem}
\label{thm:lhs-geometry}
Let $d\ge2$ be fixed.  Then
\[
R_N^{\mathrm{LHS}}=\Theta_{\PP}(N^{-2/d}),
\qquad
h_\infty(P_N^{\mathrm{LHS}})
=\Theta_{\PP}\left(\left(\frac{\log N}{N}\right)^{1/d}\right),
\]
and consequently
\[
\rho_\infty(P_N^{\mathrm{LHS}})
=\Theta_{\PP}\left(N^{1/d}(\log N)^{1/d}\right).
\]
\end{theorem}

We prove the minimum-distance and covering-radius estimates separately.
The mesh-ratio estimate then follows immediately.

\begin{proof}[Proof of the minimum-distance estimate]
Fix distinct sample indices $i$ and $\ell$.  In each coordinate,
$(X_{i,j},X_{\ell,j})$ has the distribution considered in
Lemma~\ref{lem:lhs-one-coordinate-close}, and these $d$ coordinate pairs are
independent.  Hence
\[
 \PP\bigl(\lVert\boldsymbol{X}_i-\boldsymbol{X}_\ell\rVert_\infty\le r\bigr)
 =p_N(r)^d
 \le (2r)^d.
\]
For $c>0$, put $r_N=cN^{-2/d}$.  By the union bound,
\[
\PP(R_N^{\mathrm{LHS}}\le r_N)
 \le \binom{N}{2}\,p_N(r_N)^d
 \le 2^{d-1}c^d.
\]
Letting $c\downarrow0$ gives
\[
R_N^{\mathrm{LHS}}=\Omega_{\PP}(N^{-2/d}).
\]

For the reverse bound, by relabeling the sample indices in
\eqref{eq:lhs-construction}, we may assume that $\pi_1=\mathrm{id}$.  The
relabeled jitters remain independent and uniform, and
$\pi_2,\dots,\pi_d$ remain independent uniformly distributed permutations.
For an integer $1\le H\le N/2$, let
\[
\mathcal E_H
=
\bigl\{\{i,\ell\}:0\le i<\ell<N,\ |i-\ell|\le H\bigr\}.
\]
Then
\[
 |\mathcal E_H|
 =HN-\frac{H(H+1)}2
 \asymp NH,
\]
and, for a uniform permutation $\pi$ and a fixed pair $i\ne\ell$,
\[
 p_H
 :=\PP\bigl(|\pi(i)-\pi(\ell)|\le H\bigr)
 =\frac{2|\mathcal E_H|}{N(N-1)}
 \asymp \frac HN.
\]
For $e=\{i,\ell\}\in\mathcal E_H$, let $X_e$ indicate that
$|\pi_j(i)-\pi_j(\ell)|\le H$ for $j=2,\dots,d$, and put
\[
Z_H=\sum_{e\in\mathcal E_H}X_e.
\]
If $Z_H\ge1$, then the
corresponding two sample points are at distance less than $(H+1)/N$ in every
coordinate.  Hence Lemma~\ref{lem:probabilistic-tools}(ii) gives
\begin{equation}\label{eq:lhs-min-distance-probability}
 \PP\left(R_N^{\mathrm{LHS}}<\frac{H+1}{N}\right)
 \ge \PP(Z_H\ge1)
 =1-\PP(Z_H=0)
 \ge 1-\frac{\operatorname{Var}(Z_H)}{(\EE[Z_H])^2}.
\end{equation}
Thus it suffices to estimate $\EE[Z_H]$ and $\operatorname{Var}(Z_H)$.  We first note that
\[
 \EE[Z_H]
 =|\mathcal E_H|p_H^{d-1}
 \asymp_d H^dN^{2-d}.
\]
For the variance, write
\[
 \operatorname{Var}(Z_H)
 =\sum_{e\in\mathcal E_H}\operatorname{Var}(X_e)
 +2\sum_{\substack{e,f\in\mathcal E_H\\ e<f}}
 \operatorname{Cov}(X_e,X_f).
\]
The diagonal sum is at most $\EE[Z_H]$.  We split the remaining pairs according
to whether $e$ and $f$ share an endpoint.

Suppose first that $e=\{i,\ell\}$ and $f=\{i,k\}$ with
$\ell\ne k$.  For one permutation, after fixing $\pi(i)$, each of
$\pi(\ell)$ and $\pi(k)$ has at most $2H$ admissible values.  Hence the two
closeness conditions hold with probability $O((H/N)^2)$.  Since the
permutations $\pi_2,\ldots,\pi_d$ are independent,
$\EE[X_eX_f]=O_d((H/N)^{2(d-1)})$.
Each edge has at most $O(H)$ other edges sharing an endpoint, while
$|\mathcal E_H|=O(NH)$.  Thus there are $O(NH^2)$ such pairs, whose total
contribution is
\begin{equation}\label{eq:lhs-variance-shared-endpoint}
 O_d\left(NH^2(H/N)^{2(d-1)}\right)
 =O_d\left(N^{-1}(\EE[Z_H])^2\right).
\end{equation}

Now suppose that $e=\{i,\ell\}$ and $f=\{k,m\}$ are disjoint.  For one
uniform permutation $\pi$,
\[
 \PP\bigl(|\pi(i)-\pi(\ell)|\le H,\ |\pi(k)-\pi(m)|\le H\bigr)
 =\frac{(2|\mathcal E_H|)^2+O(|\mathcal E_H|H)}{N(N-1)(N-2)(N-3)}
 =p_H^2\bigl(1+O(N^{-1})\bigr).
\]
Hence independence of $\pi_2,\ldots,\pi_d$ gives
\[
 \EE[X_eX_f]
 =p_H^{2(d-1)}\bigl(1+O_d(N^{-1})\bigr),
 \qquad
 \operatorname{Cov}(X_e,X_f)
 =O_d\left(N^{-1}p_H^{2(d-1)}\right).
\]
There are $O(|\mathcal E_H|^2)$ disjoint pairs, so their total contribution is
\begin{equation}\label{eq:lhs-variance-disjoint}
 O_d\left(N^{-1}|\mathcal E_H|^2p_H^{2(d-1)}\right)
 =O_d\left(N^{-1}(\EE[Z_H])^2\right).
\end{equation}
Combining the diagonal bound $\sum_e\operatorname{Var}(X_e)\le\EE[Z_H]$
with \eqref{eq:lhs-variance-shared-endpoint} and
\eqref{eq:lhs-variance-disjoint} gives
\[
 \operatorname{Var}(Z_H)
 \le C_d\left(\EE[Z_H]+N^{-1}(\EE[Z_H])^2\right).
\]
Fix $C\ge1$ and set
$H=\lfloor C N^{1-2/d}\rfloor$.  Then
$\EE[Z_H]\asymp_d C^d$, and the preceding variance estimate gives
\[
\frac{\operatorname{Var}(Z_H)}{(\EE[Z_H])^2}
\le C_d\left(C^{-d}+N^{-1}\right).
\]
Applying \eqref{eq:lhs-min-distance-probability} therefore yields
\[
\limsup_{N\to\infty}
\PP\left(R_N^{\mathrm{LHS}}>(C+1)N^{-2/d}\right)
\le C_dC^{-d}.
\]
Letting $C\to\infty$ gives
\[
R_N^{\mathrm{LHS}}=O_{\PP}(N^{-2/d}).\qedhere
\]
\end{proof}

\begin{proof}[Proof of the covering-radius estimate]
For the upper bound, let $Q=\prod_{j=1}^d I_j$, where each
$I_j$ is a union of $r_j$ consecutive Latin hypercube strata, and put
$a_j=r_j/N$.  Write
\[
 B_{i,j}=\chi(X_{i,j}\in I_j),
 \qquad
 J_i=\prod_{j=1}^d B_{i,j},
 \qquad 0\le i<N.
\]
For each $j$, $(B_{i,j})_{0\le i<N}$ is the indicator vector of a uniform
$r_j$-subset and is negatively associated
by~\cite[Theorem~2.11]{JoagDevProschan1983}.  The coordinate families are
independent, and negative association is preserved under independent unions
and coordinatewise monotone functions of disjoint blocks
by~\cite[Properties~P6 and~P7]{JoagDevProschan1983}.  Hence $(J_i)_{0\le i<N}$ is
negatively associated.  The same is true of $(1-J_i)_{0\le i<N}$, and the
marginal probability bound~\cite[Proposition~4]{DubhashiRanjan1998} gives
\begin{equation}\label{eq:lhs-empty-box}
\begin{aligned}
 \PP(P_N^{\mathrm{LHS}}\cap Q=\emptyset)
 &=\PP(J_i=0\text{ for all }i)\\
 &\le\prod_{i=0}^{N-1}\PP(J_i=0)
 =\left(1-\prod_{j=1}^d a_j\right)^N
 \le\exp\left(-N\prod_{j=1}^d a_j\right).
\end{aligned}
\end{equation}
Take $a=(\log N/N)^{1/d}$ and partition each coordinate axis into blocks of
consecutive strata, each containing between $\lceil aN\rceil$ and
$2\lceil aN\rceil$ strata.  Let $\mathcal Q_N$ be the resulting collection of
boxes.  Then $|\mathcal Q_N|\le C_d a^{-d}$ and every $Q\in\mathcal Q_N$
satisfies $a_j\ge a$.  Hence \eqref{eq:lhs-empty-box} and the union bound give
\[
 \PP\bigl(\exists Q\in\mathcal Q_N:\ P_N^{\mathrm{LHS}}\cap Q=\emptyset\bigr)
 \le C_d a^{-d}e^{-Na^d}
 =\frac{C_d}{\log N}
 \longrightarrow0.
\]
The boxes have diameters $O_d(a)$, which proves the upper bound.

For the lower bound, we look for an empty cube of side length of order
$(\log N/N)^{1/d}$.  Choose a constant $c>0$, whose value will be specified
below, and put
\[
 a=c\left(\frac{\log N}{N}\right)^{1/d},
 \qquad r=\lfloor aN\rfloor,
 \qquad \alpha=\frac rN,
 \qquad K=\left\lfloor\frac Nr\right\rfloor.
\]
Relabeling the sample indices in \eqref{eq:lhs-construction}, assume that
$\pi_1=\mathrm{id}$.  For $0\le s<K$, define
\[
 S_s=\{sr,\ldots,(s+1)r-1\}
 \quad\text{and}\quad
 Q_s=\left[\frac{sr}{N},\frac{(s+1)r}{N}\right)
       \times[0,\alpha)^{d-1}.
\]
The cubes $Q_s$ are disjoint and have side length $\alpha$.  For
$2\le j\le d$, let
\[
 A_j=\{i:\pi_j(i)<r\},
 \qquad
 T=A_2\cap\cdots\cap A_d,
 \qquad M=|T|.
\]
Then
\begin{equation}\label{eq:lhs-empty-block-equivalence}
 P_N^{\mathrm{LHS}}\cap Q_s=\emptyset
 \quad\Longleftrightarrow\quad
 T\cap S_s=\emptyset.
\end{equation}
Define the number of empty cubes by
\[
 Z=\sum_{s=0}^{K-1}
   \chi(P_N^{\mathrm{LHS}}\cap Q_s=\emptyset)
  =\sum_{s=0}^{K-1}
   \chi(T\cap S_s=\emptyset).
\]
If $Z\ge1$, the center of an empty cube is at $\ell^\infty$-distance at
least $\alpha/2$ from every sample point.  Hence it suffices to prove
$\PP(Z=0)\to0$.  For all sufficiently large $N$,
$a/2\le\alpha\le a$ and $K\ge(2\alpha)^{-1}$.
The sets $A_2,\ldots,A_d$ are independent uniform $r$-subsets, so
$\mu:=\EE[M]=N\alpha^{d-1}$.  Since $d\ge2$, we have $\mu\to\infty$.
By the same negative-association closure property used above, the indicators
$(\chi(i\in T))_{0\le i<N}$ are negatively associated.  Since negative
association implies pairwise nonpositive
covariances~\cite{JoagDevProschan1983},
\[
 \operatorname{Var}(M)
 \le \sum_{i=0}^{N-1}
       \operatorname{Var}\bigl(\chi(i\in T)\bigr)
 \le \EE[M]=\mu.
\]
Therefore
\begin{equation}\label{eq:lhs-intersection-size}
 \PP(M>2\mu)\le \mu^{-1}.
\end{equation}

Given $M=v$, the set $T$ is uniform over all $v$-subsets of
$\{0,\ldots,N-1\}$ by permutation invariance.  Thus, for every $s$ and
$0\le v\le N$,
\[
 p_v:=\PP(T\cap S_s=\emptyset\mid M=v)
 =\frac{\binom{N-r}{v}}{\binom Nv}
 =\prod_{u=0}^{v-1}\left(1-\frac{r}{N-u}\right).
\]
For $v\le2\mu$ and all sufficiently large $N$, we have $v\le N/2$ and
\[
 0\le \frac{r}{N-u}\le \frac{r}{N-v}\le\frac12,
 \qquad 0\le u<v.
\]
Using $\log(1-x)\ge-2x$ for $0\le x\le1/2$, together with
$r=N\alpha$, $v\le 2N\alpha^{d-1}$, and $N-v\ge N/2$, we obtain
\[
 \log p_v\ge -2\sum_{u=0}^{v-1}\frac{r}{N-u}
 \ge -\frac{2rv}{N-v}
 \ge -\frac{2(N\alpha)(2N\alpha^{d-1})}{N/2}=-8N\alpha^d.
\]
Hence, by \eqref{eq:lhs-intersection-size}, for every $s$,
\[
 \PP(T\cap S_s=\emptyset)
 =\sum_{v=0}^N p_v\PP(M=v)
 \ge \sum_{v\le2\mu}\exp(-8N\alpha^d)\PP(M=v)
 \ge (1-\mu^{-1})\exp(-8N\alpha^d).
\]

Summing over $s$ gives
\[
 \EE[Z]
 =\sum_{s=0}^{K-1}\PP(T\cap S_s=\emptyset)
 \ge \frac{1-\mu^{-1}}{2a}\exp(-8Na^d)
 =\frac{1-\mu^{-1}}{2c}
   N^{1/d-8c^d}(\log N)^{-1/d}.
\]
Choose $c$ so that $8c^d<1/d$.  Then $\EE[Z]\to\infty$.

The indicators $(\chi(i\notin T))_{0\le i<N}$ are negatively
associated.  Since the sets $S_0,\ldots,S_{K-1}$ are disjoint, the closure
properties for coordinatewise monotone functions of disjoint
blocks~\cite[Properties~P6 and~P7]{JoagDevProschan1983} show that the
empty-block indicators
$\chi(T\cap S_s=\emptyset)=\prod_{i\in S_s}\chi(i\notin T)$,
$0\le s<K$, are negatively associated.  Consequently,
$\operatorname{Var}(Z)\le \EE[Z]$.
By Lemma~\ref{lem:probabilistic-tools}(ii),
\[
 \PP(Z=0)
 \le \frac{\operatorname{Var}(Z)}{\EE[Z]^2}
 \le \frac1{\EE[Z]}
 \longrightarrow0.
\]
Thus, with probability tending to one, $Z\ge1$, and therefore
\[
 h_\infty(P_N^{\mathrm{LHS}})\ge \frac{\alpha}{2}\ge\frac a4.
\]
This proves the covering-radius lower bound.\qedhere

\end{proof}

\begin{remark}
When $d=2$, the upper bound $R_N^{\mathrm{LHS}}=O_{\PP}(N^{-1})$ also follows,
after rescaling and accounting for the within-stratum jitters,
from~\cite[Theorem~1]{BlackburnHombergerWinkler2019} on the minimum Manhattan
distance of a uniform random permutation graph.  Our argument treats all fixed
$d\ge2$ in a unified way.
\end{remark}

\section{Full Owen scrambling}\label{sec:owen}

\subsection{The one-dimensional case}

In one dimension, full Owen scrambling coincides in distribution at each level
with jittered sampling, and hence we have the following.

\begin{theorem}
\label{thm:one-dimensional-owen}
Let $b\ge2$, let $\mathcal S=(x_j)_{j\ge0}$ be a $(0,1)$-sequence in base $b$,
and let $\widetilde{\mathcal S}=(\widetilde x_j)_{j\ge0}$ be obtained by
applying a single full Owen scrambling to $\mathcal S$.  Set
$\widetilde P_n=(\widetilde x_j)_{0\le j<n}$ for $n\ge1$.
Then, as $m\to\infty$,
\[
 b^{3m/2}R_\infty(\widetilde P_{b^m})
 \xrightarrow{\mathrm d} W_{2,1/2}.
\]
Moreover, there are constants $C_b,c_b>0$ such that, almost surely,
\[
 R_\infty(\widetilde P_n)
 \le C_b n^{-3/2}(\log n)^{1/2},
 \qquad
 \rho_\infty(\widetilde P_n)
 \ge c_b n^{1/2}(\log n)^{-1/2}
\]
for all sufficiently large $n$.
\end{theorem}

\begin{proof}
Since the scrambled point set has the same distribution as a one-dimensional
jittered sample, the distributional assertion is the one-dimensional case of
Theorem~\ref{thm:jittered-minimum-distance}.
Applying \eqref{eq:jittered-elementary-tails} with
$\delta_m=Cm^{1/2}b^{-m/2}$ and choosing $C$ sufficiently large gives a
summable upper bound for
$\PP(R_\infty(\widetilde P_{b^m})>\delta_m b^{-m})$.
Hence the first Borel--Cantelli lemma gives, almost surely,
\[
 R_\infty(\widetilde P_{b^m})
 \le C_bm^{1/2}b^{-3m/2}
\]
for all sufficiently large $m$.  The two almost-sure bounds for all
sufficiently large $n$ now follow from
Lemma~\ref{lem:initial-segment-comparison} and
\eqref{eq:universal-geometric-bounds}.
\end{proof}

We now turn to arbitrary fixed dimension.

\subsection{Close pairs from isotropic occupancy}
Let $b\ge2$ and $d\ge1$ be fixed.  The close-pair construction below does
not require the full net property.  It only uses the fact that, at a level
comparable with the natural spacing scale $N^{-1/d}$, every isotropic
$b$-adic cell is occupied.

We use the following standard subtree-independence property of full Owen
scrambling.  For references, see~\cite{Owe95} and
compare~\cite[Section~2]{WiartLemieuxDong2021}.

\begin{lemma}
\label{lem:owen-prefix-selection}
Let $\boldsymbol{x}^{(1)},\ldots,\boldsymbol{x}^{(R)}\in[0,1)^d$ be fixed,
and let $\boldsymbol{X}^{(1)},\ldots,\boldsymbol{X}^{(R)}$ be their images
under a $d$-dimensional full Owen scrambling.  Fix
$k\ge0$, a set
$\mathcal I\subseteq\{1,\ldots,R\}\times\{1,\ldots,d\}$,
and integers $p_{r,j}\in\{0,\ldots,b^k-1\}$ for $(r,j)\in\mathcal I$.
Suppose that the map $(r,j)\mapsto (j,p_{r,j})$ is injective on
$\mathcal I$.  Conditional on
\[
 \left\{\left\lfloor b^kX_j^{(r)}\right\rfloor=p_{r,j}
       \text{ for all }(r,j)\in\mathcal I\right\},
\]
whenever this event has positive probability, the variables
$U_{r,j}:=b^kX_j^{(r)}-p_{r,j}$, $(r,j)\in\mathcal I$,
are independent and uniform on $[0,1)$.
The same conclusion remains valid if the indices are selected after the first
$k$ scrambled digits of all points have been revealed, provided that the
selection depends only on these digits and the original indices.  Indeed, such
a selection is independent of all permutations below level $k$.
\end{lemma}

\begin{theorem}
\label{thm:owen-upper}
Let $\gamma\in\mathbb N_0$ be fixed.  For each sufficiently large $m$, set
$k=\lfloor m/d\rfloor-\gamma$,
and let $P_{b^m}$ be a deterministic $N$-point set in $[0,1)^d$, where $N=b^m$.
Suppose that every isotropic $b$-adic cube
\begin{equation}\label{eq:isotropic-occupancy-cubes}
 Q_{\boldsymbol{a}}
 =\prod_{j=1}^d
 \left[\frac{a_j}{b^k},\frac{a_j+1}{b^k}\right),
 \qquad 0\le a_j<b^k,
\end{equation}
contains at least one point of $P_{b^m}$.  Let $\widetilde P_{b^m}$ be obtained from
$P_{b^m}$ by full Owen scrambling.

Then there exist positive constants $c_{b,d,\gamma}$ and
$C_{b,d,\gamma}$ such that, for every sufficiently large $m$ and every
$1\le\lambda\le N^{1/d}$, with probability at least
$1-C_{b,d,\gamma}\lambda^{-1}$,
\[
\begin{aligned}
R_\infty(\widetilde P_{b^m})
&\le C_{b,d,\gamma}\,N^{-3/(2d)}\lambda^{1/2},\\
\rho_\infty(\widetilde P_{b^m})
&\ge c_{b,d,\gamma}\,N^{1/(2d)}\lambda^{-1/2}.
\end{aligned}
\]
In particular,
\[
R_\infty(\widetilde P_{b^m})
=O_{\PP}\!\left(N^{-3/(2d)}\right),
\qquad
\rho_\infty(\widetilde P_{b^m})
=\Omega_{\PP}\!\left(N^{1/(2d)}\right).
\]
\end{theorem}

\begin{proof}
Put $L=\lfloor b^k/2\rfloor$.
For $\boldsymbol{r}\in\{0,\ldots,L-1\}^d$ and
$\varepsilon\in\{0,1\}$, define
\[
 I_{\boldsymbol{r}}^{(\varepsilon)}
 =\prod_{j=1}^d
 \left[\frac{2r_j+\varepsilon}{b^k},
       \frac{2r_j+\varepsilon+1}{b^k}\right).
\]
The closures of $I_{\boldsymbol{r}}^{(0)}$ and
$I_{\boldsymbol{r}}^{(1)}$ meet at one vertex.  Full Owen scrambling maps
the level-$k$ cube labels bijectively, so every such cube is nonempty.  In
each cube choose, for example, the scrambled point having the least original
index, and denote it by
$\boldsymbol{X}_{\boldsymbol{r}}^{(\varepsilon)}$.

Put $\eta=\lambda^{1/2}N^{-1/(2d)}$.
Since $1\le\lambda\le N^{1/d}$, we have $0<\eta\le1$.  Let
$E_{\boldsymbol{r}}$ be the event that, in every coordinate,
$\boldsymbol{X}_{\boldsymbol{r}}^{(0)}$ and
$\boldsymbol{X}_{\boldsymbol{r}}^{(1)}$ lie within relative distance $\eta$
of their common boundary, and put
\[
Z_m:=\sum_{\boldsymbol{r}\in\{0,\ldots,L-1\}^d}
\chi(E_{\boldsymbol{r}}).
\]
If $Z_m\ge1$, then one of the selected pairs has distance at most
\[
 2\eta b^{-k}
 \le C_{b,d,\gamma}N^{-3/(2d)}\lambda^{1/2}.
\]
Therefore Lemma~\ref{lem:probabilistic-tools}(ii) gives
\[
 \PP\left(R_\infty(\widetilde P_{b^m})
 \le C_{b,d,\gamma}N^{-3/(2d)}\lambda^{1/2}\right)
 \ge \PP(Z_m\ge1)
 \ge 1-\frac{\operatorname{Var}(Z_m)}{(\EE[Z_m])^2}.
\]
Thus it suffices to prove
\begin{equation}\label{eq:owen-upper-second-moment-target}
 \frac{\operatorname{Var}(Z_m)}{(\EE[Z_m])^2}
 \le C_{b,d,\gamma}\lambda^{-1}.
\end{equation}

We now estimate the first and second moments of $Z_m$.  Condition on the
first $k$ scrambled digits.  The selected point indices and their level-$k$
prefixes are then fixed.  The $2d$ prefixes involved in
$E_{\boldsymbol{r}}$ are distinct, so Lemma~\ref{lem:owen-prefix-selection}
shows that, under the resulting conditional law,
$\PP(E_{\boldsymbol{r}})=\eta^{2d}$.  Since this value does not depend on the
fixed prefixes, the same identity holds unconditionally.  Thus
$\EE[Z_m]=(L\eta^2)^d$.

Under the same conditioning, consider two indices
$\boldsymbol{r},\boldsymbol{r}'$.  If $r_j=r_j'$, the joint event imposes
restrictions within the same two prefixes in coordinate $j$.  Within each
prefix, this joint restriction is contained in a single tail condition of
probability $\eta$, so the resulting factor is at most $\eta^2$.  If
$r_j\ne r_j'$, the joint event imposes restrictions within four distinct
prefixes, giving the factor $\eta^4$ by
Lemma~\ref{lem:owen-prefix-selection}.  Summing over
$\boldsymbol{r}$ and $\boldsymbol{r}'$ coordinate by coordinate shows, for
every fixed realization of the first $k$ scrambled digits, that the second
moment is at most $\bigl(L\eta^2+L(L-1)\eta^4\bigr)^d$.
Averaging over the first $k$ scrambled digits therefore gives
\[
 \EE[Z_m^2]
 \le
 \bigl(L\eta^2+L(L-1)\eta^4\bigr)^d.
\]
Substituting these estimates into the left-hand side of
\eqref{eq:owen-upper-second-moment-target} gives
\[
 \frac{\operatorname{Var}(Z_m)}{\EE[Z_m]^2}
 \le \left(1-\frac1L+\frac{1}{L\eta^2}\right)^d-1
 \le \frac{C_{b,d,\gamma}}{L\eta^2}
 \le C_{b,d,\gamma}\lambda^{-1}.
\]
Here the last two inequalities use
$L\eta^2\asymp_{b,d,\gamma}\lambda$, which follows from
$k=\lfloor m/d\rfloor-\gamma$, $L^d\asymp_{b,d,\gamma}N$, and
$\eta^2=\lambda N^{-1/d}$.
This proves \eqref{eq:owen-upper-second-moment-target}, and hence the bound for
$R_\infty$.  Together with the covering lower bound in
\eqref{eq:universal-geometric-bounds}, this gives the asserted mesh-ratio
bound.
\end{proof}

\begin{corollary}
\label{cor:owen-upper-net}
Let $t,d,b$ be fixed, let $P_{b^m}$ be a deterministic $(t,m,d)$-net in base
$b$, and let $\widetilde P_{b^m}$ be obtained by full Owen scrambling.
Then the conclusions of Theorem~\ref{thm:owen-upper} hold with constants
depending only on $b,d,t$.
\end{corollary}

\begin{proof}
Take $\gamma=\lceil t/d\rceil$ and
$k=\lfloor m/d\rfloor-\gamma$.  Then $dk\le m-t$, and every level-$k$
isotropic cube contains exactly $b^{m-dk}$ points by partitioning it into
elementary intervals of volume $b^{t-m}$.  Hence
Theorem~\ref{thm:owen-upper} applies.
\end{proof}

The same estimate yields a one-sided almost-sure conclusion for a single
scrambling of a whole sequence, without the coincidence and cutoff assumptions
introduced below.

\begin{corollary}
\label{cor:owen-sequence-nonqu}
Let $\mathcal S=(\boldsymbol{x}_n)_{n\ge0}$ be a $(t,d)$-sequence in base $b$, not
necessarily digital, and let
$\widetilde{\mathcal S}=(\widetilde{\boldsymbol{x}}_n)_{n\ge0}$ be obtained by applying
a single full Owen scrambling to $\mathcal S$.  For $n\ge1$, set
$\widetilde P_n=(\widetilde{\boldsymbol{x}}_j)_{0\le j<n}$.
Then, for every $\varepsilon>0$, almost surely there exists $n_0$ such that
\[
\rho_\infty(\widetilde P_n)
\ge
c_{\varepsilon,b,d,t}\,n^{1/(2d)}(\log n)^{-1/2}
\bigl(\log(2+\log n)\bigr)^{-(1+\varepsilon)/2}
\]
for all $n\ge n_0$.  In particular, the mesh ratio of the initial segments of
$\widetilde{\mathcal S}$ diverges almost surely, and hence
$\widetilde{\mathcal S}$ is almost surely not quasi-uniform.
\end{corollary}

\begin{proof}
First let $n=b^m$.  For all sufficiently large $m$,
$\lambda_m:=m\bigl(\log(2+m)\bigr)^{1+\varepsilon}$ satisfies
$\lambda_m\le b^{m/d}$.  Applying
Corollary~\ref{cor:owen-upper-net} with $\lambda=\lambda_m$, the exceptional
probabilities are summable.  Hence the first Borel--Cantelli lemma gives,
almost surely for all sufficiently large $m$,
\[
R_\infty(\widetilde P_{b^m})
\le
C_{\varepsilon,b,d,t}\,b^{-3m/(2d)}m^{1/2}
\bigl(\log(2+m)\bigr)^{(1+\varepsilon)/2}.
\]
The asserted bound for all sufficiently large $n$ now follows from
Lemma~\ref{lem:initial-segment-comparison} and
\eqref{eq:universal-geometric-bounds}.
\end{proof}

\subsection{Coincidences and common prefixes}
\label{subsec:prefix-vectors}

The preceding argument constructs a close pair using only isotropic occupancy,
but that property alone gives no lower bound for the minimum distance.  For
example, in base $2$, a $(1,m,1)$-net places two points in each dyadic interval
of length $2^{1-m}$, and these two points may share arbitrarily many further
binary digits.  Full Owen scrambling preserves their common-prefix
length, so such a pair remains close after scrambling.  We therefore need
additional control of common-prefix coincidences.

The pair-count framework below is closely related to that of Wiart, Lemieux and
Dong~\cite{WiartLemieuxDong2021}.  Their normalized quantities count pairs
sharing prescribed initial digits, and their analysis decomposes the
conditional law of a scrambled pair according to its exact common-prefix
shell.  Here we use the same two-point structure for the minimum over all
pairs, but work primarily with the unnormalized cumulative counts.

We first record the common-prefix vector and its two-point small-ball law.

\begin{definition}[Common-prefix length]\label{def:common-prefix-length}
For $x,y\in[0,1)$, use the regular base-$b$ expansions, namely those not
ending in an infinite string of digits $b-1$, and define
\[
\ell(x,y)
:=
\sup\bigl\{\ell\in\mathbb N_0:
\lfloor b^\ell x\rfloor=\lfloor b^\ell y\rfloor\bigr\}
\in\mathbb N_0\cup\{\infty\}.
\]
Since the base $b$ is fixed throughout this section, we suppress it from the
notation.  Thus, for regular expansions, $\ell(x,y)$ is the number of common
leading base-$b$ digits and is infinite exactly when $x=y$.  For digital
points, we instead use their specified output-digit expansions, since these
determine the tree on which Owen scrambling acts.  The common-prefix length is
then infinite exactly when the specified digit sequences coincide.
For $\boldsymbol{x},\boldsymbol{y}\in[0,1)^d$, write
\[
\boldsymbol{\ell}(\boldsymbol{x},\boldsymbol{y})
=
\bigl(\ell(x_1,y_1),\ldots,\ell(x_d,y_d)\bigr).
\]
\end{definition}

The next lemma is the $\ell^\infty$ small-ball specialization of the conditional
two-point law for digital scrambling.  A related formulation appears
in~\cite[Section~2]{WiartLemieuxDong2021}.  We include the short calculation in the
normalization needed below.

\begin{lemma}\label{lem:owen-two-point}
Let $\boldsymbol{x},\boldsymbol{y}\in[0,1)^d$ be deterministic points with
common-prefix vector
$\boldsymbol{\ell}(\boldsymbol{x},\boldsymbol{y})
=(\ell_1,\ldots,\ell_d)\in(\mathbb N_0\cup\{\infty\})^d$.
Apply a $d$-dimensional full Owen scrambling, and
denote the images by
$\widetilde{\boldsymbol{x}},\widetilde{\boldsymbol{y}}$.  Then, for every
$r>0$,
\[
\PP\bigl(\lVert\widetilde{\boldsymbol{x}}-\widetilde{\boldsymbol{y}}\rVert_\infty\le r\bigr)
\le
\prod_{j=1}^d \min\{1,(b^{\ell_j+1}r)^2\},
\]
where the factor corresponding to $\ell_j=\infty$ is interpreted as $1$.
Moreover, if every $\ell_j$ is finite and
$r\le b^{-1-\max_j\ell_j}$, then
\[
\PP\bigl(\lVert\widetilde{\boldsymbol{x}}-\widetilde{\boldsymbol{y}}\rVert_\infty\le r\bigr)
=b^{2|\boldsymbol{\ell}|+d}r^{2d}.
\]
\end{lemma}

\begin{proof}
First consider $d=1$.  If $\ell(x,y)=\infty$, the two output
digit sequences coincide, and the asserted estimate is the trivial bound by
$1$.  Otherwise write $\ell=\ell(x,y)$ and put $L=b^{-\ell-1}$.  The two level-$(\ell+1)$ output children form a uniform
ordered pair of distinct children, hence are adjacent with probability $2/b$.
Given adjacency, their normalized tails are independent uniform variables, and
for $0\le r\le L$ the required boundary event has area $(r/L)^2/2$.  Thus, for $0\le r\le L$,
\[
\PP(|\widetilde x-\widetilde y|\le r)
=\frac{2}{b}\frac{(r/L)^2}{2}=b^{2\ell+1}r^2
\le (b^{\ell+1}r)^2.
\]
For $r>L$, one has
$\min\{1,(b^{\ell+1}r)^2\}=1$, so the trivial probability bound gives
the same asserted estimate.  Applying it in each coordinate and using the
independence of the coordinate scramblings proves both $d$-dimensional
statements.
\end{proof}

We now count pairs that lie in the same level-$\boldsymbol{k}$ elementary
interval.  Let $P_{b^m}=(\boldsymbol{x}_n)_{0\le n<N}$, where $N=b^m$, be an
$N$-point set.  For $\boldsymbol{k}\in\mathbb N_0^d$ and
$E\in\mathcal E_{\boldsymbol{k}}$, put
\[
\nu_m(E)
=
\#\{0\le n<N:\boldsymbol{x}_n\in E\}.
\]
For digital points, membership at a $b$-adic boundary is determined by the
specified output-digit expansion.  Define the cumulative coincidence count
\begin{equation}\label{eq:coincidence-count}
M_m(\boldsymbol{k})
:=
\sum_{E\in\mathcal E_{\boldsymbol{k}}}
\nu_m(E)\bigl(\nu_m(E)-1\bigr).
\end{equation}
Thus $M_m(\boldsymbol{k})$ counts ordered pairs of different indices lying in
the same level-$\boldsymbol{k}$ elementary interval.

\begin{definition}[Elementary-interval coincidence bound]
\label{def:uniform-coincidence}
An $N$-point set $P_{b^m}$ satisfies the elementary-interval
coincidence bound with constant $A\ge1$ if
\begin{equation}\label{eq:uniform-coincidence}
M_m(\boldsymbol{k})
\le A N(N-1)b^{-|\boldsymbol{k}|}
\qquad
\text{for all }\boldsymbol{k}\in\mathbb N_0^d.
\end{equation}
A family $(P_{b^m})$ satisfies the uniform coincidence bound if the same constant
$A$ works for every $m$.
\end{definition}

We also impose a common-prefix cutoff on the total common-prefix depth.

\begin{definition}[Common-prefix cutoff]
\label{def:uniform-prefix-cutoff}
An $N$-point set $P_{b^m}$, with $N=b^m$, satisfies the
common-prefix cutoff with parameter $\kappa\in\mathbb N_0$ if every
elementary interval $E\in\mathcal E_{\boldsymbol{k}}$ with
$|\boldsymbol{k}|\ge m+\kappa$ contains at most one point of $P_{b^m}$, counted
with multiplicity.  A family $(P_{b^m})$ satisfies a uniform common-prefix cutoff if the same
$\kappa$ works for every $m$.  Equivalently, for every pair of distinct
indices, all coordinatewise common-prefix lengths are finite and their sum is
strictly less than $m+\kappa$.
\end{definition}

\begin{lemma}
\label{lem:coincidence-zero-net}
\begin{romanenumerate}
\item Every $(0,m,d)$-net in base $b$, not necessarily digital, satisfies
Definition~\ref{def:uniform-coincidence} with $A=1$ and
Definition~\ref{def:uniform-prefix-cutoff} with $\kappa=0$.

\item Let $\boldsymbol{e}=(e_1,\ldots,e_d)\in\mathbb N^d$, and let
$\mathcal S=(\boldsymbol{x}_n)_{n\ge0}$ be a
$(0,\boldsymbol{e},d)$-sequence in base $b$, not necessarily digital.  Put
$t=\sum_{j=1}^d(e_j-1)$ and
$P_{b^m}=(\boldsymbol{x}_n)_{0\le n<b^m}$.
Then $(P_{b^m})_{m\ge1}$ satisfies the uniform coincidence bound with
$A=b^t$ and the uniform common-prefix cutoff with $\kappa=t$.
\end{romanenumerate}
\end{lemma}

\begin{proof}
For part~(i), put $D=|\boldsymbol{k}|$.  If $D\le m$, every elementary
interval in $\mathcal E_{\boldsymbol{k}}$ contains exactly $b^{m-D}$ points,
and hence
\[
M_m(\boldsymbol{k})
=
N\bigl(b^{m-D}-1\bigr)
\le N(N-1)b^{-D}.
\]
If $D\ge m$, each interval in $\mathcal E_{\boldsymbol{k}}$ is contained in
an elementary interval of volume $b^{-m}$, which contains exactly one point.
Thus $M_m(\boldsymbol{k})=0$, and the same one-point property gives the
common-prefix cutoff with $\kappa=0$.

For part~(ii), fix $m$ and $\boldsymbol{k}$, put $D=|\boldsymbol{k}|$, and set
 $q_j=e_j\lfloor k_j/e_j\rfloor$ and $Q=|\boldsymbol{q}|$.
Then $\boldsymbol{q}\le\boldsymbol{k}$ coordinatewise and $D-t\le Q\le D$.
Hence every pair counted by $M_m(\boldsymbol{k})$ lies in the same
level-$\boldsymbol{q}$ elementary interval.

Suppose first that $Q\le m$.  Partition the first $b^m$ indices into
$b^{m-Q}$ consecutive index intervals of length $b^Q$.  By the
$(0,\boldsymbol{e},d)$-sequence property, every level-$\boldsymbol{q}$
elementary interval contains exactly one point from each such interval, hence exactly
$b^{m-Q}$ points of $P_{b^m}$.  Therefore
\[
\begin{aligned}
M_m(\boldsymbol{k})
&\le N\bigl(b^{m-Q}-1\bigr)\\
&\le N(N-1)b^{-Q}
\le b^tN(N-1)b^{-D}.
\end{aligned}
\]
If $Q>m$, every level-$\boldsymbol{q}$ elementary interval contains exactly
one point among the first $b^Q$ points, and hence at most one point of $P_{b^m}$.
Thus $M_m(\boldsymbol{k})=0$.  This proves the coincidence bound with $A=b^t$.
Finally, if $D\ge m+t$, then $Q\ge m$.  The same argument shows that every
level-$\boldsymbol{k}$ elementary interval contains at most one point of
$P_{b^m}$.  Hence the cutoff holds with $\kappa=t$.
\end{proof}

\subsection{Minimum-distance lower tail}

\begin{theorem}
\label{thm:owen-lower-tail}
Let $m\ge1$, and let $P_{b^m}=(\boldsymbol{x}_n)_{0\le n<N}$ be an $N$-point set, where
$N=b^m$, satisfying Definitions~\ref{def:uniform-coincidence}
and~\ref{def:uniform-prefix-cutoff} with constants $A$ and $\kappa$, and let
$\widetilde P_{b^m}=(\widetilde{\boldsymbol{x}}_n)_{0\le n<N}$ be obtained by
full Owen scrambling.  Put
$R_m:=R_\infty(\widetilde P_{b^m})$.  Then, for every $r>0$,
\begin{align}
\PP(R_m\le r)
&\le
\frac12
\sum_{|\boldsymbol{\ell}|<m+\kappa}
M_m(\boldsymbol{\ell})
\prod_{j=1}^d
\min\{1,(b^{\ell_j+1}r)^2\}
\label{eq:prefix-vector-tail}\\
&\le
C_{b,d,A,\kappa}N^3(1+m)^{d-1}r^{2d}.
\notag
\end{align}
\end{theorem}

\begin{proof}
For an ordered pair $(n,n')$, let
$\boldsymbol{\ell}=(\ell(x_{n,1},x_{n',1}),\ldots,
\ell(x_{n,d},x_{n',d}))$.  By the equivalent pairwise formulation of the
common-prefix cutoff, $\boldsymbol{\ell}\in\mathbb N_0^d$ and
$|\boldsymbol{\ell}|<m+\kappa$.
For each such vector $\boldsymbol{\ell}$, the number of ordered pairs whose common-prefix vector is exactly $\boldsymbol{\ell}$ is at most
$M_m(\boldsymbol{\ell})$.  Lemma~\ref{lem:owen-two-point}, followed by the
union bound, therefore gives \eqref{eq:prefix-vector-tail}.

Using the coincidence bound~\eqref{eq:uniform-coincidence} and
$\min\{1,(b^{\ell_j+1}r)^2\}\le b^{2\ell_j+2}r^2$,
we obtain
\[
\begin{aligned}
\PP(R_m\le r)
&\le
\frac{A N(N-1)}2
\sum_{|\boldsymbol{\ell}|<m+\kappa}
 b^{-|\boldsymbol{\ell}|}
\prod_{j=1}^d b^{2\ell_j+2}r^2\\
&=
\frac{A b^{2d}}2N(N-1)r^{2d}
\sum_{|\boldsymbol{\ell}|<m+\kappa}
 b^{-|\boldsymbol{\ell}|}b^{2|\boldsymbol{\ell}|}\\
&=
\frac{A b^{2d}}2N(N-1)r^{2d}
\sum_{D=0}^{m+\kappa-1}
\binom{D+d-1}{d-1}b^D\\
&\le C_{b,d,A,\kappa}N^3(1+m)^{d-1}r^{2d}.
\end{aligned}
\]
This proves the second inequality in
\eqref{eq:prefix-vector-tail}.
\end{proof}

\begin{corollary}
\label{cor:owen-probabilistic-bounds}
Let $b,d,t$ be fixed, and let $(P_{b^m})_{m\ge t}$ be a family of deterministic
$(t,m,d)$-nets in base $b$, not necessarily digital.  Put $N=b^m$.  Assume that the
family satisfies the uniform coincidence bound with constant $A$ and the
uniform common-prefix cutoff with parameter $\kappa$.  If $\widetilde P_{b^m}$ is
obtained by full Owen scrambling, then
\[
\begin{aligned}
R_\infty(\widetilde P_{b^m})
&=\Omega_{\PP}\!\left(
N^{-3/(2d)}(\log N)^{-(d-1)/(2d)}
\right),
&
R_\infty(\widetilde P_{b^m})
&=O_{\PP}\!\left(N^{-3/(2d)}\right),
\end{aligned}
\]
\[
h_\infty(\widetilde P_{b^m})\asymp_{b,d,t}N^{-1/d},
\]
and
\[
\begin{aligned}
\rho_\infty(\widetilde P_{b^m})
&=\Omega_{\PP}\!\left(N^{1/(2d)}\right),
&
\rho_\infty(\widetilde P_{b^m})
&=O_{\PP}\!\left(
N^{1/(2d)}(\log N)^{(d-1)/(2d)}
\right).
\end{aligned}
\]
\end{corollary}

\begin{proof}
In Corollary~\ref{cor:owen-upper-net}, choose $\lambda$ to be a sufficiently
large fixed constant.  This gives
\[
R_\infty(\widetilde P_{b^m})=O_{\PP}\!\left(N^{-3/(2d)}\right).
\]
Likewise, applying Theorem~\ref{thm:owen-lower-tail} with
$r=N^{-3/(2d)}(1+m)^{-(d-1)/(2d)}\tau^{-1}$
gives an exceptional probability at most
$C_{b,d,A,\kappa}\tau^{-2d}$.  Choosing $\tau$ to be a sufficiently large
fixed constant yields
\[
R_\infty(\widetilde P_{b^m})
=\Omega_{\PP}\!\left(
N^{-3/(2d)}(1+m)^{-(d-1)/(2d)}
\right).
\]
Since $m=\log_b N$, this is the asserted lower bound.  Scrambling preserves
the net property, so \eqref{eq:universal-geometric-bounds} and
\eqref{eq:net-covering} give
$h_\infty(\widetilde P_{b^m})\asymp_{b,d,t}N^{-1/d}$.  Combining these covering
bounds with the minimum-distance bounds gives the two mesh-ratio bounds.
\end{proof}

\begin{corollary}
\label{cor:owen-badic-as}
Let $\mathcal S=(\boldsymbol{x}_n)_{n\ge0}$ be a $(t,d)$-sequence in
base $b$, not necessarily digital, and suppose that, for every $m$, its first
$b^m$ points satisfy the uniform coincidence bound and the uniform common-prefix
cutoff with constants $A$ and $\kappa$ independent of $m$.  Apply a single
full Owen scrambling to the whole sequence, and set
$\widetilde P_n=(\widetilde{\boldsymbol{x}}_j)_{0\le j<n}$ for $n\ge1$.
Then, almost surely,
\[
\frac{\log R_\infty(\widetilde P_n)}{\log n}
\longrightarrow -\frac{3}{2d},
\qquad
\frac{\log\rho_\infty(\widetilde P_n)}{\log n}
\longrightarrow \frac{1}{2d}
\]
as $n\to\infty$.  In particular, the scrambled sequence is almost surely not
quasi-uniform.
\end{corollary}

\begin{proof}
First let $n=b^m$.  Apply Theorem~\ref{thm:owen-lower-tail} with
$r=b^{-3m/(2d)}(1+m)^{-(d-1)/(2d)}m^{-1}$
and Corollary~\ref{cor:owen-upper-net} with $\lambda=m^2$.
For all sufficiently large $m$, we have $1\le\lambda\le b^{m/d}$, and the two
exceptional probabilities are $O(m^{-2d}+m^{-2})$, hence summable.  The first
Borel--Cantelli lemma therefore gives, almost surely,
\[
\log R_\infty(\widetilde P_{b^m})
=-\frac{3}{2d}\log b^m+O(\log m).
\]
Moreover, scrambling preserves the $(t,m,d)$-net property, so
$h_\infty(\widetilde P_{b^m})\asymp_{b,d,t}b^{-m/d}$.
Since $\log m=o(m)$, Lemma~\ref{lem:initial-segment-comparison}
transfers the exponents of $R_\infty$ and $h_\infty$ from $n=b^m$ to all
$n$.  This gives the asserted limits.
\end{proof}

\begin{remark}
Let $\mathcal H$ be a Halton sequence in pairwise coprime bases
$b_1,\ldots,b_d$.  Apply a single full Owen scrambling to each coordinate in
its respective base, independently across coordinates, and let
$\widetilde P_N$ denote the first $N$ scrambled points.  Then the conclusions
of Corollary~\ref{cor:owen-probabilistic-bounds} hold as $N\to\infty$, with
constants depending only on $d$ and the bases.  The conclusions of
Corollary~\ref{cor:owen-badic-as} hold for the whole scrambled sequence.
For $\boldsymbol{k}\in\mathbb N_0^d$, put
$B(\boldsymbol{k})=\prod_{j=1}^d b_j^{k_j}$.
By the Chinese remainder theorem, every level-$\boldsymbol{k}$ mixed-base
elementary interval contains either
$\lfloor N/B(\boldsymbol{k})\rfloor$ or
$\lceil N/B(\boldsymbol{k})\rceil$ of the first $N$ Halton points.  Hence the
corresponding coincidence count is at most
$N(N-1)B(\boldsymbol{k})^{-1}$, and it vanishes when
$B(\boldsymbol{k})\ge N$.  Moreover,
\[
\sum_{B(\boldsymbol{k})<N}B(\boldsymbol{k})
\ll_{b_1,\ldots,b_d}N(1+\log N)^{d-1}.
\]
These estimates give the same lower-tail bound as
Theorem~\ref{thm:owen-lower-tail}.  Choosing
$k_j=\lfloor\log_{b_j}N^{1/d}\rfloor$ in the occupancy argument gives the
same close-pair and covering estimates.  The almost-sure conclusions for all
initial segments then follow from
Lemma~\ref{lem:initial-segment-comparison} applied along a geometric
subsequence.
\end{remark}

\subsection{Digital-net specialization}

For digital point sets, the cumulative coincidence count has an algebraic
kernel representation.  Its integrality makes the common-prefix cutoff an
automatic consequence of the coincidence bound~\eqref{eq:uniform-coincidence}.

Suppose that $b$ is a prime power, and write the digital point set as
$P_{b^m}=(\boldsymbol{x}_i)_{0\le i<N}$, using its specified output-digit
expansions.  Identify the indices $0,\ldots,N-1$ with their $m$ base-$b$
digits, and let $\oplus$ denote digitwise addition.  For
$\boldsymbol{k}\in\mathbb N_0^d$, let
\[
V_m(\boldsymbol{k})
=
\left\{0\le i<N:
\ell\bigl(0,x_{i,j}\bigr)\ge k_j
\text{ for }j=1,\ldots,d\right\}.
\]
By the digital group structure,
\[
\ell\bigl(x_{i,j},x_{i\oplus h,j}\bigr)
=
\ell\bigl(0,x_{h,j}\bigr),
\qquad j=1,\ldots,d,
\quad 0\le i,h<N.
\]
Thus $\boldsymbol{x}_i$ and $\boldsymbol{x}_{i\oplus h}$ lie in the same
level-$\boldsymbol{k}$ elementary interval precisely when
$h\in V_m(\boldsymbol{k})$.  Every $h$
occurs in exactly $N$ ordered pairs $(i,i\oplus h)$, and $h=0$ accounts for
the diagonal pairs.  Hence
\begin{equation}\label{eq:digital-coincidence-kernel}
M_m(\boldsymbol{k})
=
N\bigl(|V_m(\boldsymbol{k})|-1\bigr)
\qquad
\text{for all }\boldsymbol{k}\in\mathbb N_0^d.
\end{equation}

\begin{lemma}
\label{lem:digital-prefix-cutoff}
Let $P_{b^m}$ be digital and satisfy
Definition~\ref{def:uniform-coincidence} with constant $A$.  Then $P_{b^m}$
satisfies Definition~\ref{def:uniform-prefix-cutoff} with parameter
$\kappa=\lceil\log_b A\rceil$.
\end{lemma}

\begin{proof}
If $|\boldsymbol{k}|\ge m+\lceil\log_b A\rceil$, then
\eqref{eq:digital-coincidence-kernel} and the coincidence bound~\eqref{eq:uniform-coincidence} give
$|V_m(\boldsymbol{k})|-1\le A(N-1)b^{-|\boldsymbol{k}|}<1$.
The left-hand side is a nonnegative integer, so $V_m(\boldsymbol{k})=\{0\}$.
Thus the prefix map at level $\boldsymbol{k}$ is injective, so every
level-$\boldsymbol{k}$ elementary interval contains at most one point of $P_{b^m}$,
and the required cutoff follows.
\end{proof}

\begin{corollary}
\label{cor:owen-digital-specialization}
Let $b,d,t$ be fixed, and let $(P_{b^m})_{m\ge t}$ be a family of deterministic
digital $(t,m,d)$-nets in base $b$.  If the family satisfies the uniform coincidence
bound with a constant $A$ independent of $m$, then all conclusions of
Corollary~\ref{cor:owen-probabilistic-bounds} hold, with
$\kappa=\lceil\log_b A\rceil$.
\end{corollary}

\begin{proof}
Lemma~\ref{lem:digital-prefix-cutoff} supplies the required uniform
common-prefix cutoff, so Corollary~\ref{cor:owen-probabilistic-bounds} applies.
\end{proof}

\section{Matrix and linear scrambling}
\label{sec:linear}

This section studies matrix-based randomizations in base $2$.  We first prove
a general one-sided estimate for matrix and linear scrambling in dimension
$d\ge2$.  We then introduce a balanced-prefix affine-tail model that permits
sharp two-sided analysis and show that one-dimensional matrix and linear
scrambling both fit this model.

\begin{remark}
The restriction to base $2$ comes from
Lemmas~\ref{lem:exact-digital-distance}
and~\ref{lem:small-digital-difference-count}, which use the independent-sign
representation of a binary digitwise difference.  The lower-triangular
algebra extends to finite fields, but for $b>2$ a different estimate for small
real digit differences is required.
\end{remark}

Throughout this section, $\oplus$ denotes addition in $\mathbb F_2$ and its
componentwise extension to binary vectors.  For a binary sequence
$\boldsymbol{y}=(y_r)_{r\ge1}$, write
$[\boldsymbol{y}]:=\sum_{r\ge1}y_r2^{-r}$.
For a finite vector $\boldsymbol{v}\in\mathbb F_2^M$, we use the same notation
and set $[\boldsymbol{v}]:=\sum_{r=1}^M v_r2^{-r}$.  Equivalently, we regard
$\boldsymbol{v}$ as being zero-padded.
Let $C_1,\ldots,C_d\in\mathbb F_2^{\mathbb N\times m}$ be the generating
matrices, with finite matrices extended by zero rows.  For
$\boldsymbol{a}\in\mathbb F_2^m$, set
\[
x_{\boldsymbol{a},j}=[C_j\boldsymbol{a}],
\qquad
\boldsymbol{x}_{\boldsymbol{a}}
=(x_{\boldsymbol{a},1},\ldots,x_{\boldsymbol{a},d}).
\]
We index the resulting digital point set by its input vectors and write
\[
P_{2^m}=(\boldsymbol{x}_{\boldsymbol{a}})_{\boldsymbol{a}\in\mathbb F_2^m}.
\]
\subsection{Digital distance}

For $\boldsymbol{D}\in\mathbb F_2^M$, put
\[
\mu_M(\boldsymbol{D})
=
\min_{\boldsymbol{v}\in\mathbb F_2^M}
\bigl|[\boldsymbol{v}\oplus\boldsymbol{D}]-[\boldsymbol{v}]\bigr|.
\]

\begin{lemma}
\label{lem:exact-digital-distance}
For $\boldsymbol{D}\in\mathbb F_2^M\setminus\{\boldsymbol{0}\}$, let
$\nu(\boldsymbol{D}):=\min\{r:D_r=1\}$.
Then
\begin{equation}\label{eq:exact-digital-distance}
\mu_M(\boldsymbol{D})
=
2^{-\nu(\boldsymbol{D})}
-
\sum_{\substack{r>\nu(\boldsymbol{D})\\D_r=1}}2^{-r}.
\end{equation}
\end{lemma}

\begin{proof}
For fixed $\boldsymbol{D}$, varying $\boldsymbol{v}$ allows the signs in
\[
[\boldsymbol{v}\oplus\boldsymbol{D}]-[\boldsymbol{v}]
=
\sum_{r:D_r=1}\varepsilon_r2^{-r},
\qquad \varepsilon_r\in\{-1,1\},
\]
to be chosen independently.  The leading term $2^{-\nu(\boldsymbol{D})}$
exceeds the sum of all lower binary weights, so the minimum absolute value is
obtained by taking the opposite sign for every lower nonzero digit.
\end{proof}

\begin{lemma}
\label{lem:small-digital-difference-count}
For every integer $K\ge1$,
\begin{enumerate}
\item[(i)]
for every $\nu\in\{1,\ldots,M\}$,
\[
\#\{\boldsymbol{D}\in\mathbb F_2^M:
 \nu(\boldsymbol{D})=\nu,\ \mu_M(\boldsymbol{D})\le K2^{-M}\}
=
\min\{K,2^{M-\nu}\}.
\]
\item[(ii)]
\[
\#\{\boldsymbol{D}\in\mathbb F_2^M\setminus\{\boldsymbol{0}\}:
\mu_M(\boldsymbol{D})\le K2^{-M}\}
\le MK.
\]
\end{enumerate}
\end{lemma}

\begin{proof}
By Lemma~\ref{lem:exact-digital-distance}, for fixed
$\nu(\boldsymbol{D})=\nu$, the value of $\mu_M(\boldsymbol{D})$ runs once through
$2^{-M},2\cdot2^{-M},\ldots,2^{M-\nu}2^{-M}$
as the lower digits vary.  This proves (i).  Summing over $\nu$ proves
(ii).\qedhere
\end{proof}

\subsection{Higher-dimensional scrambling}\label{subsec:higher-dimensional-matrix-linear}

The next lemma isolates the probabilistic estimate for one fixed nonzero
input-vector difference.

\begin{lemma}
\label{lem:linear-fixed-difference}
Let $C_1,\ldots,C_d\in\mathbb F_2^{\mathbb N\times m}$ generate the binary
digital point set
$P_{2^m}=(\boldsymbol{x}_{\boldsymbol{a}})_{\boldsymbol{a}\in\mathbb F_2^m}$,
and fix a digital shift $\boldsymbol{\delta}$.  Let
$(\boldsymbol{x}_{\boldsymbol{a}}^{\mathrm{mat},\boldsymbol{\delta}})_{
\boldsymbol{a}\in\mathbb F_2^m}$ be its matrix-scrambled version obtained using
independent random infinite nonsingular lower triangular matrices
$L_1,\ldots,L_d$.

For $\boldsymbol{h}\in\mathbb F_2^m\setminus\{\boldsymbol{0}\}$, set
\[
\ell_{j,m}(\boldsymbol{h})
:=
\min\{\ell(0,x_{\boldsymbol{h},j}),m-1\},
\qquad
\boldsymbol{\ell}_m(\boldsymbol{h})
:=
\bigl(\ell_{1,m}(\boldsymbol{h}),\ldots,
      \ell_{d,m}(\boldsymbol{h})\bigr).
\]
Then, for every integer $K\ge1$,
\[
\PP\left(
\min_{\boldsymbol{a}\in\mathbb F_2^m}
\left\lVert
\boldsymbol{x}_{\boldsymbol{a}\oplus\boldsymbol{h}}^{\mathrm{mat},\boldsymbol{\delta}}
-
\boldsymbol{x}_{\boldsymbol{a}}^{\mathrm{mat},\boldsymbol{\delta}}
\right\rVert_\infty
<K2^{-m}
\right)
\le
K^d2^{|\boldsymbol{\ell}_m(\boldsymbol{h})|+d-dm}.
\]
The estimate is uniform over the choice of the fixed shift.
\end{lemma}

\begin{proof}
Let $C_{j,m}$ denote the first $m$ rows of $C_j$, and put
$\boldsymbol{D}_j(\boldsymbol{h})
:=C_{j,m}\boldsymbol{h}\in\mathbb F_2^m$.
Let $L_{j,m}$ be the leading $m\times m$ submatrix of $L_j$, and let
$\boldsymbol{\delta}_j^{(m)}$ be the first $m$ shift digits.  The first $m$
digits of the $j$-th scrambled coordinate are
$\boldsymbol{v}_{\boldsymbol{a},j}
:=L_{j,m}C_{j,m}\boldsymbol{a}
\oplus\boldsymbol{\delta}_j^{(m)}$.  Hence
$\boldsymbol{v}_{\boldsymbol{a}\oplus\boldsymbol{h},j}
\oplus\boldsymbol{v}_{\boldsymbol{a},j}
=L_{j,m}\boldsymbol{D}_j(\boldsymbol{h})$.  Write
$q_{\boldsymbol{a},j}
:=\sum_{r=1}^m(\boldsymbol{v}_{\boldsymbol{a},j})_r2^{m-r}$.
Thus the $j$-th scrambled coordinate belongs to the closed dyadic interval
$2^{-m}[q_{\boldsymbol{a},j},q_{\boldsymbol{a},j}+1]$.

Suppose that, for some $\boldsymbol{a}\in\mathbb F_2^m$,
\[
\left\lVert
\boldsymbol{x}_{\boldsymbol{a}\oplus\boldsymbol{h}}^{\mathrm{mat},\boldsymbol{\delta}}
-
\boldsymbol{x}_{\boldsymbol{a}}^{\mathrm{mat},\boldsymbol{\delta}}
\right\rVert_\infty
<K2^{-m}.
\]
Then, for every coordinate $j$,
$|q_{\boldsymbol{a}\oplus\boldsymbol{h},j}
-q_{\boldsymbol{a},j}|\le K$.
Since
$\boldsymbol{v}_{\boldsymbol{a}\oplus\boldsymbol{h},j}
=\boldsymbol{v}_{\boldsymbol{a},j}
\oplus L_{j,m}\boldsymbol{D}_j(\boldsymbol{h})$,
the definition of $\mu_m$ gives
\[
\mu_m(L_{j,m}\boldsymbol{D}_j(\boldsymbol{h}))
\le
2^{-m}
\left|
q_{\boldsymbol{a}\oplus\boldsymbol{h},j}
-q_{\boldsymbol{a},j}
\right|
\le K2^{-m}.
\]
If $\boldsymbol{D}_j(\boldsymbol{h})\ne\boldsymbol{0}$, its first nonzero digit
is at position $r_0=\ell_{j,m}(\boldsymbol{h})+1$, and
$L_{j,m}\boldsymbol{D}_j(\boldsymbol{h})$ is uniform over the $2^{m-r_0}$ vectors
with first nonzero digit at $r_0$.  Lemma~\ref{lem:small-digital-difference-count}(i)
therefore gives
\[
\PP\left(
\mu_m(L_{j,m}\boldsymbol{D}_j(\boldsymbol{h}))\le K2^{-m}
\right)
=
\frac{\min\{K,2^{m-r_0}\}}{2^{m-r_0}}
\le K2^{\ell_{j,m}(\boldsymbol{h})+1-m}.
\]
If $\boldsymbol{D}_j(\boldsymbol{h})=\boldsymbol{0}$, the same bound is trivial
because its right-hand side is at least one.  Independence across the
coordinates gives the result.
\end{proof}

The theorem below uses only digitality and the coincidence bound~\eqref{eq:uniform-coincidence} for the
minimum-distance estimate.  The $(t,m,d)$-net property is needed only for the
covering radius.  The restriction $d\ge2$ is essential because the final
union bound has no decaying factor when $d=1$.

\begin{theorem}
\label{thm:linear-general-polylog}
Let $d\ge2$ and $t\ge0$ be fixed.  For each $m\ge\max\{t,1\}$, let
$P_{2^m}=(\boldsymbol{x}_{\boldsymbol{a}})_{\boldsymbol{a}\in\mathbb F_2^m}$
be a digital $(t,m,d)$-net in base $2$.  Assume that the family $(P_{2^m})$
satisfies the uniform coincidence bound with a constant $A$ independent of
$m$.  Put $N=2^m$.  There are constants
$c,C>0$, depending only on $d,t,A$, such that, for every
$m\ge\max\{t,1\}$ and every
$u\ge1$ satisfying $u\log N\le N^{1-1/d}$,
\[
\PP\left(
 R_\infty(P_{2^m}^{\mathrm{scr}})<
 c\,N^{-1/d}(u\log N)^{-1}
 \text{ or }
 \rho_\infty(P_{2^m}^{\mathrm{scr}})>C\,u\log N
\right)
\le C\,u^{-d}.
\]
In particular,
\[
R_\infty(P_{2^m}^{\mathrm{scr}})
=\Omega_{\PP}\!\left(N^{-1/d}(\log N)^{-1}\right),
\qquad
\rho_\infty(P_{2^m}^{\mathrm{scr}})=O_{\PP}(\log N).
\]
The bounds for matrix scrambling are uniform over the fixed shift
$\boldsymbol{\delta}$.
\end{theorem}

\begin{proof}
For $\boldsymbol{h}\in\mathbb F_2^m\setminus\{\boldsymbol{0}\}$, let
$\boldsymbol{\ell}_m(\boldsymbol{h})$ be as in
Lemma~\ref{lem:linear-fixed-difference}.  Fix an integer $K\ge1$.  By that
lemma and a union bound over the nonzero input-vector differences,
\begin{equation}\label{eq:linear-small-distance-union}
\PP\left(
R_\infty(P_{2^m}^{\mathrm{mat},\boldsymbol{\delta}})<K2^{-m}
\right)
\le
\sum_{\boldsymbol{h}\in\mathbb F_2^m\setminus\{\boldsymbol{0}\}}
K^d2^{|\boldsymbol{\ell}_m(\boldsymbol{h})|+d-dm}.
\end{equation}

For $\boldsymbol{q}\in\{0,\ldots,m-1\}^d$, let
\[
H_m(\boldsymbol{q})
=
\{\boldsymbol{h}\in\mathbb F_2^m\setminus\{\boldsymbol{0}\}:
\ell_{j,m}(\boldsymbol{h})=q_j,\ j=1,\ldots,d\}.
\]
Identify each integer in $V_m(\boldsymbol{q})$ with its length-$m$
binary digit vector.  Then
$H_m(\boldsymbol{q})\subseteq
V_m(\boldsymbol{q})\setminus\{\boldsymbol{0}\}$, and
\eqref{eq:digital-coincidence-kernel} together with the coincidence bound~\eqref{eq:uniform-coincidence}
gives
\[
|H_m(\boldsymbol{q})|
\le |V_m(\boldsymbol{q})|-1
\le A(N-1)2^{-|\boldsymbol{q}|}.
\]
The sets $H_m(\boldsymbol{q})$ partition
$\mathbb F_2^m\setminus\{\boldsymbol{0}\}$, and the summand in
\eqref{eq:linear-small-distance-union} depends on $\boldsymbol{h}$ only through
$\boldsymbol{\ell}_m(\boldsymbol{h})$.  Hence, grouping the sum by
$\boldsymbol{q}=\boldsymbol{\ell}_m(\boldsymbol{h})$ gives
\[
\begin{aligned}
\PP\left(
R_\infty(P_{2^m}^{\mathrm{mat},\boldsymbol{\delta}})<K2^{-m}
\right)
&\le
\sum_{\boldsymbol{q}\in\{0,\ldots,m-1\}^d}
|H_m(\boldsymbol{q})|
K^d2^{|\boldsymbol{q}|+d-dm}\\
&\le
\sum_{\boldsymbol{q}\in\{0,\ldots,m-1\}^d}
A(N-1)2^{-|\boldsymbol{q}|}
K^d2^{|\boldsymbol{q}|+d-dm}\\
&=
A m^d(N-1)K^d2^{d-dm}\\
&\le
A2^d m^dK^d2^{-m(d-1)}.
\end{aligned}
\]
In the last two lines, the factors $2^{-|\boldsymbol{q}|}$ and
$2^{|\boldsymbol{q}|}$ cancel, there are $m^d$ possible vectors $\boldsymbol{q}$, and $N-1<2^m$.

With $N=2^m$, take
$K=\lfloor N^{1-1/d}/(u\log N)\rfloor$.
The assumed range of $u$ gives $K\ge1$ and
$K\ge N^{1-1/d}/(2u\log N)$.
Substituting this choice of $K$ into the preceding estimate, and using
$m=(\log N)/\log 2$, gives
\[
\begin{aligned}
\PP\left(
R_\infty(P_{2^m}^{\mathrm{mat},\boldsymbol{\delta}})
<\frac12\,N^{-1/d}(u\log N)^{-1}
\right)
&\le
\PP\left(
R_\infty(P_{2^m}^{\mathrm{mat},\boldsymbol{\delta}})<K2^{-m}
\right)\\
&\le C_{d,A}u^{-d}.
\end{aligned}
\]
By the $(t,m,d)$-net property and the endpoint convention above,
\eqref{eq:net-covering} gives
$h_\infty(P_{2^m}^{\mathrm{mat},\boldsymbol{\delta}})
\le C_{d,t}N^{-1/d}$.
Consequently,
\[
\left\{
\rho_\infty(P_{2^m}^{\mathrm{mat},\boldsymbol{\delta}})
>4C_{d,t}\,u\log N
\right\}
\subseteq
\left\{
R_\infty(P_{2^m}^{\mathrm{mat},\boldsymbol{\delta}})
<\frac12\,N^{-1/d}(u\log N)^{-1}
\right\},
\]
which proves the stated probability bound.  Taking $u=B$ and then choosing
$B$ sufficiently large gives the two asserted probabilistic orders, uniformly
over fixed digital shifts.  Averaging every fixed-shift estimate over the
random digital shift proves the corresponding assertions for
$P_{2^m}^{\mathrm{lin}}$.
\end{proof}

\begin{corollary}
\label{cor:linear-sequence-polylog}
Let $d\ge2$, and let $\mathcal S=(\boldsymbol{x}_n)_{n\ge0}$ be a digital
$(t,d)$-sequence in base $2$.  Assume that the point sets consisting of the
first $2^m$ points satisfy the uniform coincidence bound with a constant $A$
independent of $m$.  Then, for every $\varepsilon>0$, almost surely,
\[
\rho_\infty(P_n^{\mathrm{scr}})
\le
C_{\varepsilon,d,t,A}
(\log n)^{1+1/d}
\bigl(\log(2+\log n)\bigr)^{(1+\varepsilon)/d}
\]
for all sufficiently large $n$.  In the case of matrix scrambling, this
conclusion holds simultaneously for all fixed infinite digital shifts
$\boldsymbol{\delta}$.
\end{corollary}

\begin{proof}
First let $n=2^m$.  For all sufficiently large $m$, apply
Theorem~\ref{thm:linear-general-polylog} with
\[
u_m=m^{1/d}(\log(2+m))^{(1+\varepsilon)/d}.
\]
The sum of the exceptional probabilities is bounded by
\[
C\sum_{m\ge m_0}u_m^{-d}
=C\sum_{m\ge m_0}\frac{1}{m(\log(2+m))^{1+\varepsilon}}<\infty.
\]
In the fixed-shift proof, the bad event is contained
in the same shift-independent event for every digital shift.  The first
Borel--Cantelli lemma therefore gives, simultaneously for all shifts,
\[
R_\infty(P_{2^m}^{\mathrm{scr}})
\ge
c_{\varepsilon,d,t,A}
2^{-m/d}m^{-1-1/d}
(\log(2+m))^{-(1+\varepsilon)/d}
\]
for all sufficiently large $m$.  The same conclusion follows under the joint
law for linear scrambling.

The asserted bound for all sufficiently large $n$ now follows from
Lemma~\ref{lem:initial-segment-comparison} and the net covering bound.
\end{proof}

The preceding minimum-distance argument uses only the coincidence bound, but
it is one-sided and becomes ineffective when $d=1$.  Motivated by recent uses
of random digital-net models in integration~\cite{Pan26,GLT26}, we next
introduce an affine-tail model adapted to sharp minimum-distance analysis.
Its application to
one-dimensional matrix and linear scrambling is given in
Subsection~\ref{subsec:one-dimensional-matrix-linear}.

\subsection{Balanced-prefix affine tails}
\label{subsec:balanced-affine-tails}

\begin{definition}[Balanced-prefix random affine-tail model]
\label{def:balanced-prefix-random-affine}
Fix $d,k\ge1$ and set $m=kd$.  Let
$B_1,\ldots,B_d$ be arbitrary deterministic matrices with
$B_j\in\mathbb F_2^{k\times m}$ such that
\[
B:=
\begin{pmatrix}
B_1\\
\vdots\\
B_d
\end{pmatrix}
\in\mathbb F_2^{m\times m}
\]
is invertible.  Let
$G_1,\ldots,G_d\in\mathbb F_2^{\mathbb N\times m}$ be independent random
matrices with independent uniform entries, and set
\[
C_j=
\begin{pmatrix}
B_j\\
G_j
\end{pmatrix}
\in\mathbb F_2^{\mathbb N\times m}.
\]
For each $j=1,\ldots,d$, fix an arbitrary deterministic digital shift
\[
\boldsymbol{\delta}_j=
\begin{pmatrix}
\boldsymbol{\eta}_j\\
\boldsymbol{c}_j
\end{pmatrix}
\in\mathbb F_2^{\mathbb N},
\qquad
\boldsymbol{\eta}_j\in\mathbb F_2^k,
\quad
\boldsymbol{c}_j=(c_{j,r})_{r\ge1}\in\mathbb F_2^{\mathbb N}.
\]
For $\boldsymbol{v}\in\mathbb F_2^m$, define
$x_{\boldsymbol{v},j}
=\bigl[C_j\boldsymbol{v}\oplus\boldsymbol{\delta}_j\bigr]$ for
$j=1,\ldots,d$, and set
$P_{2^m}=(\boldsymbol{x}_{\boldsymbol{v}})_{
\boldsymbol{v}\in\mathbb F_2^m}$.
We say that $P_{2^m}$ follows the balanced-prefix random affine-tail model.
The only randomness in the model comes from the matrices
$G_1,\ldots,G_d$.  The matrices $B_1,\ldots,B_d$ and the shifts
$\boldsymbol{\delta}_1,\ldots,\boldsymbol{\delta}_d$ are fixed.
\end{definition}

The following indexed form will be convenient.

\begin{lemma}
\label{lem:balanced-prefix-indexed-form}
Fix $d,k\ge1$, set $m=kd$, and fix
\[
\boldsymbol{\eta}
=(\boldsymbol{\eta}_1^\top,\ldots,\boldsymbol{\eta}_d^\top)^\top
\in(\mathbb F_2^k)^d
\qquad\text{and}\qquad
\boldsymbol{c}_j\in\mathbb F_2^{\mathbb N},
\quad j=1,\ldots,d.
\]
Let $A_1,\ldots,A_d\in\mathbb F_2^{\mathbb N\times m}$ be independent random
matrices with independent uniform entries.  For
$\boldsymbol{a}
=(\boldsymbol{a}_1^\top,\ldots,\boldsymbol{a}_d^\top)^\top
\in(\mathbb F_2^k)^d$,
define $\boldsymbol{x}_{\boldsymbol a}$ by taking its output digit vector in
coordinate $j$ to be
\[
\begin{pmatrix}
\boldsymbol{a}_j\\
A_j(\boldsymbol{a}\oplus\boldsymbol{\eta})\oplus\boldsymbol{c}_j
\end{pmatrix},
\qquad j=1,\ldots,d.
\]
Then $P_{2^m}=(\boldsymbol{x}_{\boldsymbol a})_{
\boldsymbol a\in(\mathbb F_2^k)^d}$ follows the balanced-prefix random
affine-tail model of Definition~\ref{def:balanced-prefix-random-affine}.
\end{lemma}

\begin{proof}
Indeed, the change of variables
$\boldsymbol a=B\boldsymbol v\oplus\boldsymbol\eta$ reindexes the points in
Definition~\ref{def:balanced-prefix-random-affine}, and the matrices
$A_j:=G_jB^{-1}$ remain independent with independent uniform entries.
\end{proof}

For the indexed representation in Lemma~\ref{lem:balanced-prefix-indexed-form}
and $\ell\ge1$, let $A_{j,\ell}$ be the first $\ell$ rows of $A_j$ and put
\[
\boldsymbol{c}_{j,\ell}:=(c_{j,1},\ldots,c_{j,\ell}),
\qquad
T_{j,\ell}(\boldsymbol{a})
:=A_{j,\ell}(\boldsymbol{a}\oplus\boldsymbol{\eta})
  \oplus\boldsymbol{c}_{j,\ell}.
\]
Then $A_{j,\ell}$ is a uniform random $\ell\times m$ matrix, independently
across $j$.

\begin{lemma}
\label{lem:uniform-random-linear-map}
Let $V$ be a finite-dimensional vector space over $\mathbb F_2$, and let
$A:V\to\mathbb F_2^q$ be chosen uniformly from
$\operatorname{Hom}(V,\mathbb F_2^q)$.
\begin{romanenumerate}
\item If $\boldsymbol{v}_1,\ldots,\boldsymbol{v}_r\in V$ are linearly
independent, then
$A\boldsymbol{v}_1,\ldots,A\boldsymbol{v}_r$ are independent and uniform on
$\mathbb F_2^q$.
\item Let $\mathcal U\subset V$ be an affine subspace of dimension $s\ge q$,
let $\boldsymbol{c}\in\mathbb F_2^q$, and put
$T(\boldsymbol{a})=A\boldsymbol{a}\oplus\boldsymbol{c}$.  Then, for every
$\boldsymbol{y}\in\mathbb F_2^q$,
\[
\PP\{\boldsymbol{y}\notin T(\mathcal U)\}
\le 2^{q-s}.
\]
\end{romanenumerate}
\end{lemma}

\begin{proof}
Extend $\boldsymbol{v}_1,\ldots,\boldsymbol{v}_r$ to a basis of $V$.  A
uniform random linear map is obtained by assigning independent uniform images
to the basis vectors, which proves the first assertion.

Write $\mathcal U=\boldsymbol{a}_0\oplus W$, where $\dim W=s$.  If $A|_W$ is
surjective, then $T(\mathcal U)=\mathbb F_2^q$.  In bases of $W$ and
$\mathbb F_2^q$, the map $A|_W$ is a uniform random $q\times s$ matrix $B$.
Therefore
\[
\PP\{\operatorname{rank}B<q\}
\le
\sum_{\boldsymbol{z}\in\mathbb F_2^q\setminus\{\boldsymbol{0}\}}
\PP\{\boldsymbol{z}^\top B=\boldsymbol{0}\}
=(2^q-1)2^{-s}<2^{q-s},
\]
which proves the second assertion.
\end{proof}

For $0\le n<2^k$, let
\[
\operatorname{bin}_k(n)=(a_1,\ldots,a_k)\in\mathbb F_2^k,
\qquad
n=\sum_{q=1}^ka_q2^{k-q},
\]
be the usual $k$-digit binary expansion of $n$.  This vector indexes the $n$-th dyadic interval.

\begin{lemma}
\label{lem:consecutive-binary-differences}
For $0\le n<2^k-1$, one has
\[
\operatorname{bin}_k(n)\oplus\operatorname{bin}_k(n+1)
=
\boldsymbol{D}_\tau
:=(0,\ldots,0,\underbrace{1,\ldots,1}_{\tau+1})
\]
for a unique $\tau\in\{0,\ldots,k-1\}$, namely the number of trailing ones in
$\operatorname{bin}_k(n)$.  Hence the vectors
$\operatorname{bin}_k(n)\oplus\operatorname{bin}_k(n+1)$ take at most $k$
possible values.
\end{lemma}

The key point is that the vectors
$\operatorname{bin}_k(n)\oplus\operatorname{bin}_k(n+1)$ take at most $k$
values, although there are $2^k-1$ adjacent pairs of level-$k$ dyadic
intervals.  The tail difference of a pair depends only on this vector.  Hence
all pairs with difference $\boldsymbol{D}_\tau$ are governed by the same
random vectors $A_{j,\ell}\boldsymbol{D}_\tau$.  This rigidity distinguishes the model from
independent jitter and is used in the proof below.

\begin{theorem}
\label{thm:balanced-affine-tail-sharp}
Let $d\ge1$ be fixed, let $m=kd$, and let $P_{2^m}$ follow the
balanced-prefix random affine-tail model of
Definition~\ref{def:balanced-prefix-random-affine}.  There are constants
$C_d>0$ and $k_0=k_0(d)$ such that, for every $k\ge k_0$ and every
$1\le u\le k/4$,
\begin{align}
\PP\left(
 R_\infty(P_{2^m})<\frac{1}{8u}\,\frac{2^{-k}}{k}
\right)&\le C_du^{-d},
\label{eq:balanced-affine-tail-lower}\\
\PP\left(
 R_\infty(P_{2^m})>8u\,\frac{2^{-k}}{k}
\right)&\le C_du^{-d}.
\label{eq:balanced-affine-tail-upper}
\end{align}
The bounds are uniform over all fixed matrices $B_1,\ldots,B_d$ satisfying
the invertibility condition in
Definition~\ref{def:balanced-prefix-random-affine} and all digital shifts.
In particular,
\[
R_\infty(P_{2^m})
=\Theta_{\PP}\!\left(\frac{2^{-k}}{k}\right)
=\Theta_{\PP}\!\left(2^{-m/d}m^{-1}\right),
\qquad
h_\infty(P_{2^m})\asymp_d2^{-k}=2^{-m/d},
\]
and
\[
\rho_\infty(P_{2^m})=\Theta_{\PP}(k)=\Theta_{\PP}(m).
\]
\end{theorem}

\begin{proof}
We use the representation in Lemma~\ref{lem:balanced-prefix-indexed-form},
indexed by $\boldsymbol{a}\in(\mathbb F_2^k)^d$.  Fix $k$ sufficiently large and
$1\le u\le k/4$.  For the lower tail, put
$\ell_+=\lceil\log_2 k\rceil+\lceil\log_2 u\rceil$.
For every $\boldsymbol{a},\boldsymbol{a}'\in(\mathbb F_2^k)^d$,
both $\boldsymbol{\eta}$ and $\boldsymbol{c}_{j,\ell_+}$ cancel, and hence
\[
T_{j,\ell_+}(\boldsymbol{a})\oplus T_{j,\ell_+}(\boldsymbol{a}')
=A_{j,\ell_+}(\boldsymbol{a}\oplus\boldsymbol{a}').
\]
Suppose that two distinct points are at $\ell^\infty$-distance less than
$2^{-k-\ell_+-1}$.  In every coordinate, they lie in the same or adjacent
level-$k$ dyadic intervals.  Write
$\boldsymbol{a}_j=\operatorname{bin}_k(n_j)$ and
$\boldsymbol{a}'_j=\operatorname{bin}_k(n'_j)$ for the corresponding interval
indices.  Then
$|n_j-n'_j|\le1$.  If the intervals are adjacent in exactly $r$ coordinates,
Lemma~\ref{lem:consecutive-binary-differences} shows that there are at most
$\binom dr k^r$ possible values of
$\boldsymbol{a}\oplus\boldsymbol{a}'$.  Fix one such nonzero difference
$\boldsymbol{D}=(\boldsymbol{D}_1,\ldots,\boldsymbol{D}_d)$.  Since
$\boldsymbol{D}\ne\boldsymbol{0}$,
Lemma~\ref{lem:uniform-random-linear-map}(i), together with independence across
$j$, shows that the vectors $A_{j,\ell_+}\boldsymbol{D}$ are independent and
uniform on $\mathbb F_2^{\ell_+}$.  For this fixed $\boldsymbol{D}$, we bound
the probability that a pair with index difference $\boldsymbol{D}$ has
distance less than $2^{-k-\ell_+-1}$.  Let
$I(\boldsymbol{D})=\{j:\boldsymbol{D}_j\ne\boldsymbol{0}_k\}$.
Then $|I(\boldsymbol{D})|=r$, and these are exactly the coordinates in which the
level-$k$ intervals are adjacent.  In each such coordinate, the two points must
approach their common boundary from opposite sides, which requires
$A_{j,\ell_+}\boldsymbol{D}=\boldsymbol{1}_{\ell_+}$.  This condition has
probability $2^{-\ell_+}$.  In each of the remaining $d-r$ coordinates, the two
points lie in the same level-$k$ interval.  There closeness requires
$\mu_{\ell_+}(A_{j,\ell_+}\boldsymbol{D})\le2^{-\ell_+}$.
Lemma~\ref{lem:small-digital-difference-count}(ii) gives at most $\ell_++1$
admissible values, including the zero vector, so this condition has probability
at most $(\ell_++1)2^{-\ell_+}$.  Hence the probability for this fixed
$\boldsymbol{D}$ is at most
\[
2^{-r\ell_+}\bigl((\ell_++1)2^{-\ell_+}\bigr)^{d-r}
=2^{-d\ell_+}(\ell_++1)^{d-r}.
\]
Grouping the union bound by $r$ gives
\[
\begin{aligned}
\PP\{R_\infty(P_{2^m})<2^{-k-\ell_+-1}\}
&\le
2^{-d\ell_+}
\sum_{r=1}^d\binom dr k^r(\ell_++1)^{d-r}\\
&\le
2^{-d\ell_+}(k+\ell_++1)^d\\
&\le
u^{-d}\left(1+\frac{\ell_++1}{k}\right)^d
\le C_du^{-d}.
\end{aligned}
\]
The penultimate inequality follows from $2^{-\ell_+}\le (uk)^{-1}$.
Moreover, $u\le k/4$ gives $\ell_+=O(\log k)$, so the final factor is bounded
for sufficiently large $k$.  Also, $2^{-k-\ell_+-1}\ge 2^{-k}/(8uk)$.
This proves the first finite-level bound.

For the upper tail, put
\[
s=\lceil\log_2u\rceil,
\qquad
\ell_- = \left\lfloor\log_2 k\right\rfloor-s,
\qquad
H=\lfloor k/2\rfloor.
\]
The range of $u$ ensures that $\ell_-\ge1$.  For $0\le \tau\le H$, let
$\boldsymbol{D}_{\tau}$ be as in
Lemma~\ref{lem:consecutive-binary-differences}.  We consider the difference
vectors
\[
\boldsymbol{D}(\boldsymbol{\tau})
=(\boldsymbol{D}_{\tau_1},\ldots,\boldsymbol{D}_{\tau_d}),
\qquad
\boldsymbol{\tau}=(\tau_1,\ldots,\tau_d)\in\{0,\ldots,H\}^d.
\]
For a fixed $\boldsymbol{\tau}$, it is enough to find
$\boldsymbol{a}=(\boldsymbol{a}_1,\ldots,\boldsymbol{a}_d)$ such that each
$\boldsymbol{a}_j=\operatorname{bin}_k(n_j)$ indexes the left interval of an
adjacent pair, the corresponding index difference is
$\boldsymbol{D}(\boldsymbol{\tau})$, and
\[
A_{j,\ell_-}\boldsymbol{D}(\boldsymbol{\tau})
=\boldsymbol{1}_{\ell_-}
\quad\text{and}\quad
T_{j,\ell_-}(\boldsymbol{a})=\boldsymbol{1}_{\ell_-}
\qquad\text{for }j=1,\ldots,d.
\]
Indeed, the point in the corresponding right interval then has its first
$\ell_-$ tail digits equal to zero, so the two points are at distance at most
$2^{1-k-\ell_-}$.  We first find a $\boldsymbol{\tau}$ satisfying the first
condition and then realize the second condition for this
$\boldsymbol{\tau}$.  For each $\boldsymbol{\tau}$, let
\[
E_{\boldsymbol{\tau}}
=
\left\{
A_{j,\ell_-}\boldsymbol{D}(\boldsymbol{\tau})
=\boldsymbol{1}_{\ell_-}
\ \text{for }j=1,\ldots,d
\right\}.
\]
We have $\PP(E_{\boldsymbol{\tau}})=2^{-d\ell_-}$.  Since the vectors
$\boldsymbol{D}(\boldsymbol{\tau})$ are distinct and nonzero, they are linearly
independent in pairs over $\mathbb F_2$.  Hence the events
$E_{\boldsymbol{\tau}}$ are pairwise independent.  Define
\[
Z=\sum_{\boldsymbol{\tau}\in\{0,\ldots,H\}^d}
\chi(E_{\boldsymbol{\tau}}).
\]
Thus $Z$ counts these indices, and
$\EE[Z]=(H+1)^d2^{-d\ell_-}\ge2^{d(s-1)}$.
Since the indicators are pairwise independent,
\[
\operatorname{Var}(Z)
=\sum_{\boldsymbol{\tau}}
\operatorname{Var}\!\left(\chi(E_{\boldsymbol{\tau}})\right)
\le \sum_{\boldsymbol{\tau}}\EE\!\left[\chi(E_{\boldsymbol{\tau}})\right]
=\EE[Z].
\]
Lemma~\ref{lem:probabilistic-tools}(ii) gives
$\PP(Z=0)\le2^{d(1-s)}$.

On $\{Z>0\}$, let $\widehat{\boldsymbol{\tau}}$ be the lexicographically first
$\boldsymbol{\tau}\in\{0,\ldots,H\}^d$ for which $E_{\boldsymbol{\tau}}$ occurs.
It remains to realize
this difference vector by left intervals with the required tail digits.
Define
\[
\mathcal A_{\widehat{\boldsymbol{\tau}}}
=
\left\{
\boldsymbol{a}\in(\mathbb F_2^k)^d:
(a_{j,k-H},\ldots,a_{j,k})
=(\boldsymbol{0}_{H-\widehat\tau_j+1},
  \boldsymbol{1}_{\widehat\tau_j})
\ \text{for }j=1,\ldots,d
\right\}.
\]
For $\boldsymbol{a}\in\mathcal A_{\widehat{\boldsymbol{\tau}}}$, write
$\boldsymbol{a}_j=\operatorname{bin}_k(n_j)$.  Its prescribed suffix has
exactly $\widehat\tau_j$ trailing ones, so
Lemma~\ref{lem:consecutive-binary-differences} gives
$\boldsymbol{a}_j\oplus\boldsymbol{D}_{\widehat\tau_j}
=\operatorname{bin}_k(n_j+1)$.
Thus the associated level-$k$ intervals are respectively the $n_j$-th and
$(n_j+1)$-st intervals, which form an adjacent pair.  If some
$\boldsymbol{a}\in\mathcal A_{\widehat{\boldsymbol{\tau}}}$ satisfies
\begin{equation}\label{eq:upper-tail-free-digit-condition}
T_{j,\ell_-}(\boldsymbol{a})=\boldsymbol{1}_{\ell_-}
\qquad\text{for }j=1,\ldots,d,
\end{equation}
then $E_{\widehat{\boldsymbol{\tau}}}$ gives
\[
T_{j,\ell_-}
\bigl(\boldsymbol{a}\oplus
\boldsymbol{D}(\widehat{\boldsymbol{\tau}})\bigr)
=
T_{j,\ell_-}(\boldsymbol{a})
\oplus A_{j,\ell_-}\boldsymbol{D}(\widehat{\boldsymbol{\tau}})
=\boldsymbol{0}_{\ell_-}
\]
for every $j$.  The two corresponding points therefore approach the common
boundary of the adjacent intervals from the left and right and satisfy
\[
\left\lVert
\boldsymbol{x}_{\boldsymbol{a}}
-
\boldsymbol{x}_{\boldsymbol{a}\oplus
\boldsymbol{D}(\widehat{\boldsymbol{\tau}})}
\right\rVert_\infty
\le2^{1-k-\ell_-}.
\]
Consequently, on $\{Z>0\}$, the event
$\{R_\infty(P_{2^m})>2^{1-k-\ell_-}\}$ can occur only if no
$\boldsymbol{a}\in\mathcal A_{\widehat{\boldsymbol{\tau}}}$ satisfying
\eqref{eq:upper-tail-free-digit-condition} exists.

Let $U$ be the subspace of
$\boldsymbol{u}=(\boldsymbol{u}_1,\ldots,\boldsymbol{u}_d)$ for which the last
$H+1$ entries of each $\boldsymbol{u}_i$ are zero.  Let $W$ be the subspace
for which the first $k-H-1$ entries of each $\boldsymbol{u}_i$ are zero.  For
$\boldsymbol{\tau}\in\{0,\ldots,H\}^d$, set
\[
\boldsymbol{a}^{0}(\boldsymbol{\tau})
=
\bigl((\boldsymbol{0}_{k-\tau_i},\boldsymbol{1}_{\tau_i})\bigr)_{i=1}^d.
\]
Then
$\mathcal A_{\widehat{\boldsymbol{\tau}}}
=\boldsymbol{a}^{0}(\widehat{\boldsymbol{\tau}})\oplus U$, and both
$\boldsymbol{a}^{0}(\boldsymbol{\tau})$ and
$\boldsymbol{D}(\boldsymbol{\tau})$ belong to $W$.  Since every
$E_{\boldsymbol{\tau}}$ is determined by
$A_{j,\ell_-}|_W$, $j=1,\ldots,d$, so is the lexicographically first
successful $\widehat{\boldsymbol{\tau}}$.  Since
$(\mathbb F_2^k)^d=W\oplus U$, conditional on these restrictions the maps
$A_{j,\ell_-}|_U$ remain independent and uniform.  Decompose
$\boldsymbol{\eta}=\boldsymbol{\eta}_W\oplus\boldsymbol{\eta}_U$ according to
$W\oplus U$.  Since $\boldsymbol{u}\mapsto
\boldsymbol{u}\oplus\boldsymbol{\eta}_U$ is a bijection of $U$, the
$U$-component of $\boldsymbol{\eta}$ can be absorbed into
$\boldsymbol{u}$.  Hence, writing
$\boldsymbol{a}=\boldsymbol{a}^{0}(\widehat{\boldsymbol{\tau}})\oplus
\boldsymbol{u}$, the map
\[
\boldsymbol{u}
\longmapsto
\bigl(T_{1,\ell_-}(\boldsymbol{a}),\ldots,
T_{d,\ell_-}(\boldsymbol{a})\bigr)
\]
has a fixed translation and the uniform random linear part
\[
\boldsymbol{u}
\longmapsto
\bigl(A_{1,\ell_-}\boldsymbol{u},\ldots,
A_{d,\ell_-}\boldsymbol{u}\bigr)
\in\mathbb F_2^{d\ell_-}.
\]
For all sufficiently large $k$, we have
$\dim U=d(k-H-1)\ge d\ell_-$.  Lemma~\ref{lem:uniform-random-linear-map}(ii)
therefore gives
\[
\PP\left(
\nexists\,\boldsymbol{a}\in\mathcal A_{\widehat{\boldsymbol{\tau}}}
\text{ satisfying }\eqref{eq:upper-tail-free-digit-condition}
\,\middle|\,
A_{j,\ell_-}|_W,\ j=1,\ldots,d
\right)
\le2^{d\ell_--d(k-H-1)}.
\]
It follows that
\[
\PP\left(R_\infty(P_{2^m})>2^{1-k-\ell_-}\right)
\le2^{d(1-s)}+2^{d\ell_--d(k-H-1)}.
\]
Now
\[
2^{d(1-s)}\le2^du^{-d},
\qquad
2^{d\ell_--d(k-H-1)}
\le2^dk^du^{-d}2^{-dk/2}\le2^du^{-d}
\]
for all sufficiently large $k$.  Finally,
$2^{1-k-\ell_-}\le8u\,2^{-k}/k$.
This proves the second finite-level bound.

The two bounds imply
\[
R_\infty(P_{2^m})
=\Theta_{\PP}\!\left(\frac{2^{-k}}{k}\right)
=\Theta_{\PP}\!\left(2^{-m/d}m^{-1}\right).
\]
The balanced prefix gives $h_\infty(P_{2^m})\asymp_d2^{-k}$.  Hence
$\rho_\infty(P_{2^m})=\Theta_{\PP}(k)=\Theta_{\PP}(m)$.
\end{proof}

\subsection{One-dimensional scrambling}
\label{subsec:one-dimensional-matrix-linear}

\begin{corollary}
\label{cor:one-dimensional-matrix-linear}
For each $m\in\mathbb N$, let
$C^{(m)}\in\mathbb F_2^{\mathbb N\times m}$ generate a one-dimensional
binary digital $(0,m,1)$-net $P_{2^m}$.  Then, as $m\to\infty$,
\[
R_\infty(P_{2^m}^{\mathrm{scr}})=\Theta_{\PP}(2^{-m}m^{-1}),
\qquad
\rho_\infty(P_{2^m}^{\mathrm{scr}})=\Theta_{\PP}(m).
\]
Moreover, \eqref{eq:balanced-affine-tail-lower} and
\eqref{eq:balanced-affine-tail-upper} hold in both cases with $d=1$, $k=m$,
and $P_{2^m}$ replaced by $P_{2^m}^{\mathrm{scr}}$.  In the case of matrix
scrambling, the constants and bounds are uniform over the fixed shift
$\boldsymbol{\delta}$.
\end{corollary}

\begin{proof}
Fix $m$ and write $C=C^{(m)}$.  Let $C_m$ and $L_m$ be the leading
$m\times m$ submatrices of $C$ and the scrambling matrix $L$, respectively.
Both are invertible.  Conditional on $L_m$ and, in the linear scrambling
case, on the digital shift, the point set generated by $LC$ follows the
balanced-prefix random affine-tail model with $d=1$, $k=m$, and prefix matrix
$B=L_mC_m$.  Indeed, in every row below row $m$, the first $m$ entries of $L$
form an independent uniform vector.  Multiplication by the invertible matrix
$C_m$ and addition of the remaining terms preserve independence and
uniformity.  The result follows from
Theorem~\ref{thm:balanced-affine-tail-sharp}, whose bounds are uniform over
the prefix matrix and the shift.
\end{proof}

\begin{corollary}
\label{cor:one-dimensional-scrambled-sequence}
Let $\mathcal S$ be a binary digital $(0,1)$-sequence.  For every
$\varepsilon>0$, set $m_r=\lceil r^{1+\varepsilon}\rceil$ and
$N_r=2^{m_r}$.  Then, almost surely,
\[
\rho_\infty(P_{N_r}^{\mathrm{scr}})
\ge \frac{r^\varepsilon}{8(\log(2+r))^2}
\]
for all sufficiently large $r$.  Thus the mesh ratio diverges along the sparse
subsequence $(N_r)_{r\ge1}$, and both scrambled sequences are almost surely not
quasi-uniform.
\end{corollary}

\begin{proof}
Put $u_r=r(\log(2+r))^2$.  Since $m_r\ge r^{1+\varepsilon}$, one has
$1\le u_r\le m_r/4$ for all sufficiently large $r$.  The upper-tail bound in
Theorem~\ref{thm:balanced-affine-tail-sharp}, applied with $d=1$, therefore gives
\[
\PP\left(
R_\infty(P_{N_r}^{\mathrm{scr}})>8u_r\frac{2^{-m_r}}{m_r}
\right)
\le C u_r^{-1}.
\]
The right-hand side is summable.  The first Borel--Cantelli lemma therefore
gives, almost surely,
\[
R_\infty(P_{N_r}^{\mathrm{scr}})
\le 8u_r\frac{2^{-m_r}}{m_r}
\]
for all sufficiently large $r$.  The universal covering lower bound
\eqref{eq:universal-geometric-bounds} gives
$h_\infty(P_{N_r}^{\mathrm{scr}})\ge2^{-m_r-1}$, and hence
\[
\rho_\infty(P_{N_r}^{\mathrm{scr}})
=\frac{2h_\infty(P_{N_r}^{\mathrm{scr}})}
{R_\infty(P_{N_r}^{\mathrm{scr}})}
\ge\frac{m_r}{8u_r}
\ge\frac{r^\varepsilon}{8(\log(2+r))^2}.\qedhere
\]
\end{proof}

\begin{remark}
At the ordinary dyadic levels, the $u^{-1}$ tail bound is not summable for any
choice $u_m=o(m)$.  The super-geometric subsequence in the preceding corollary
makes the exceptional probabilities summable and proves
non-quasi-uniformity, but the argument does not prove almost-sure divergence
along all dyadic levels.  This is analogous to the obstruction in
\cite[Lemma~1.4]{DGLPS25} to transferring quasi-uniformity from a subsequence
with unbounded successive ratios to the full sequence.
\end{remark}

\section*{Declaration of generative AI use}
The author used OpenAI ChatGPT 5.6 Sol and Codex for exploring and checking
proof strategies, deriving and drafting parts of several proofs, conducting
literature searches, and providing editorial assistance.  All mathematical
arguments, calculations, references, and conclusions were independently
checked and verified by the author, who takes full responsibility for the
contents of the paper.

\bibliographystyle{amsplain}
\bibliography{ref}

@article{ArratiaGoldsteinGordon1989,
  author  = {Arratia, R. and Goldstein, L. and Gordon, L.},
  title   = {Two moments suffice for {Poisson} approximations: the {Chen--Stein} method},
  journal = {Annals of Probability},
  volume  = {17},
  number  = {1},
  pages   = {9--25},
  year    = {1989}
}

@article{BlackburnHombergerWinkler2019,
  author  = {Blackburn, S. R. and Homberger, C. and Winkler, P.},
  title   = {The minimum {Manhattan} distance and minimum jump of permutations},
  journal = {Journal of Combinatorial Theory, Series A},
  volume  = {161},
  pages   = {364--386},
  year    = {2019}
}

@article{DGLPS25,
  author       = {Dick, J. and Goda, T. and Larcher, G. and Pillichshammer, F. and Suzuki, K.},
  title        = {On the quasi-uniformity properties of {quasi-Monte Carlo} point sets and sequences -- {Part I}: {Lattices} and {Kronecker} sequences},
  journal      = {Mathematics of Computation},
  pages        = {to appear},
  year         = {2026}
}

@article{DGS25,
  author       = {Dick, J. and Goda, T. and Suzuki, K.},
  title        = {On the quasi-uniformity properties of {quasi-Monte Carlo} point sets and sequences -- {Part II}: {Digital} nets and sequences},
  journal      = {Mathematics of Computation},
  pages        = {to appear},
  year         = {2026}
}

@book{DP10,
  author    = {Dick, J. and Pillichshammer, F.},
  title     = {Digital Nets and Sequences: Discrepancy Theory and {quasi-Monte Carlo} Integration},
  publisher = {Cambridge University Press},
  address   = {Cambridge},
  year      = {2010}
}

@article{Doerr2022,
  author  = {Doerr, B.},
  title   = {A sharp discrepancy bound for jittered sampling},
  journal = {Mathematics of Computation},
  volume  = {91},
  number  = {336},
  pages   = {1871--1892},
  year    = {2022}
}

@article{DubhashiRanjan1998,
  author  = {Dubhashi, D. P. and Ranjan, D.},
  title   = {Balls and bins: a study in negative dependence},
  journal = {Random Structures \& Algorithms},
  volume  = {13},
  number  = {2},
  pages   = {99--124},
  year    = {1998}
}

@misc{GLT26,
  author       = {Goda, T. and Liu, Y. and Tempone, R.},
  title        = {{Quasi-Monte Carlo} with a {Hankel} random digital net},
  howpublished = {arXiv:2604.24105},
  year         = {2026}
}

@article{JoagDevProschan1983,
  author  = {Joag-Dev, K. and Proschan, F.},
  title   = {Negative association of random variables with applications},
  journal = {Annals of Statistics},
  volume  = {11},
  number  = {1},
  pages   = {286--295},
  year    = {1983}
}

@article{McKayBeckmanConover1979,
  author  = {McKay, M. D. and Beckman, R. J. and Conover, W. J.},
  title   = {A comparison of three methods for selecting values of input variables in the analysis of output from a computer code},
  journal = {Technometrics},
  volume  = {21},
  number  = {2},
  pages   = {239--245},
  year    = {1979}
}

@book{Nie92,
  author    = {Niederreiter, H.},
  title     = {Random Number Generation and {quasi-Monte Carlo} Methods},
  series    = {CBMS-NSF Regional Conference Series in Applied Mathematics},
  volume    = {63},
  publisher = {Society for Industrial and Applied Mathematics (SIAM)},
  address   = {Philadelphia, PA},
  year      = {1992}
}

@incollection{Owe95,
  author    = {Owen, A. B.},
  title     = {Randomly permuted $(t,m,s)$-nets and $(t,s)$-sequences},
  booktitle = {Monte Carlo and Quasi-Monte Carlo Methods in Scientific Computing},
  series    = {Lecture Notes in Statistics},
  volume    = {106},
  publisher = {Springer},
  address   = {New York},
  pages     = {299--317},
  year      = {1995}
}

@article{Pan26,
  author  = {Pan, Z.},
  title   = {Automatic optimal-rate convergence of randomized nets using median-of-means},
  journal = {Mathematics of Computation},
  volume  = {95},
  number  = {359},
  pages   = {1415--1446},
  year    = {2026},
  doi     = {10.1090/mcom/4093}
}

@article{PausingerSteinerberger2016,
  author  = {Pausinger, F. and Steinerberger, S.},
  title   = {On the discrepancy of jittered sampling},
  journal = {Journal of Complexity},
  volume  = {33},
  pages   = {199--216},
  year    = {2016}
}

@article{PronzatoZhigljavsky2023,
  author  = {Pronzato, L. and Zhigljavsky, A.},
  title   = {Quasi-uniform designs with optimal and near-optimal uniformity constant},
  journal = {Journal of Approximation Theory},
  volume  = {294},
  pages   = {105931},
  year    = {2023}
}

@article{ReznikovSaff2016,
  author  = {Reznikov, A. and Saff, E. B.},
  title   = {The covering radius of randomly distributed points on a manifold},
  journal = {International Mathematics Research Notices},
  number  = {19},
  pages   = {6065--6094},
  year    = {2016}
}

@incollection{SchulteThale2016,
  author    = {Schulte, M. and Th{\"a}le, C.},
  title     = {{Poisson} point process convergence and extreme values in stochastic geometry},
  booktitle = {Stochastic Analysis for Poisson Point Processes},
  series    = {Bocconi \& Springer Series},
  volume    = {7},
  publisher = {Springer},
  address   = {Cham},
  pages     = {255--294},
  year      = {2016}
}

@article{Stein1987,
  author  = {Stein, M.},
  title   = {Large sample properties of simulations using {Latin} hypercube sampling},
  journal = {Technometrics},
  volume  = {29},
  number  = {2},
  pages   = {143--151},
  year    = {1987}
}

@article{Te13,
  author  = {Tezuka, S.},
  title   = {On the discrepancy of generalized {Niederreiter} sequences},
  journal = {Journal of Complexity},
  volume  = {29},
  number  = {3--4},
  pages   = {240--247},
  year    = {2013}
}

@article{WiartLemieuxDong2021,
  author  = {Wiart, J. and Lemieux, C. and Dong, G. Y.},
  title   = {On the dependence structure and quality of scrambled $(t,m,s)$-nets},
  journal = {Monte Carlo Methods and Applications},
  volume  = {27},
  number  = {1},
  pages   = {1--26},
  year    = {2021}
}

@book{Wendland2005,
  author    = {Wendland, H.},
  title     = {Scattered Data Approximation},
  series    = {Cambridge Monographs on Applied and Computational Mathematics},
  volume    = {17},
  publisher = {Cambridge University Press},
  address   = {Cambridge},
  year      = {2005}
}

@article{SchabackWendland2006,
  author  = {Schaback, R. and Wendland, H.},
  title   = {Kernel techniques: from machine learning to meshless methods},
  journal = {Acta Numerica},
  volume  = {15},
  pages   = {543--639},
  year    = {2006}
}

@article{PronzatoMuller2012,
  author  = {Pronzato, L. and M{\"u}ller, W. G.},
  title   = {Design of computer experiments: space filling and beyond},
  journal = {Statistics and Computing},
  volume  = {22},
  number  = {3},
  pages   = {681--701},
  year    = {2012}
}

@incollection{GruenschlossHanikaSchwedeKeller2008,
  author    = {Gr{\"u}nschlo{\ss}, L. and Hanika, J. and Schwede, R. and Keller, A.},
  title     = {$(t,m,s)$-nets and maximized minimum distance},
  booktitle = {Monte Carlo and Quasi-Monte Carlo Methods 2006},
  publisher = {Springer},
  address   = {Berlin},
  pages     = {397--412},
  year      = {2008}
}

@incollection{GruenschlossKeller2009,
  author    = {Gr{\"u}nschlo{\ss}, L. and Keller, A.},
  title     = {$(t,m,s)$-nets and maximized minimum distance, {Part II}},
  booktitle = {Monte Carlo and Quasi-Monte Carlo Methods 2008},
  publisher = {Springer},
  address   = {Berlin},
  pages     = {395--409},
  year      = {2009}
}

@article{Goda2024,
  author  = {Goda, T.},
  title   = {The {Sobol'} sequence is not quasi-uniform in dimension 2},
  journal = {Proceedings of the American Mathematical Society},
  volume  = {152},
  number  = {8},
  pages   = {3209--3213},
  year    = {2024}
}

@article{GodaHoferSuzuki2026,
  author  = {Goda, T. and Hofer, R. and Suzuki, K.},
  title   = {Disproving the quasi-uniformity of the {Halton} sequences and of some {Halton}-type sequences},
  journal = {Journal of Complexity},
  volume  = {95},
  pages   = {102047},
  year    = {2026},
  doi     = {10.1016/j.jco.2026.102047}
}

@article{Matousek1998,
  author  = {Matou{\v{s}}ek, J.},
  title   = {On the {$L_2$}-discrepancy for anchored boxes},
  journal = {Journal of Complexity},
  volume  = {14},
  number  = {4},
  pages   = {527--556},
  year    = {1998}
}

@article{Owen2003,
  author  = {Owen, A. B.},
  title   = {Variance with alternative scramblings of digital nets},
  journal = {ACM Transactions on Modeling and Computer Simulation},
  volume  = {13},
  number  = {4},
  pages   = {363--378},
  year    = {2003}
}

@misc{Suzuki2025,
  author       = {Suzuki, K.},
  title        = {Exact $\ell^\infty$-separation radius of {Sobol'} sequences in dimension 2},
  howpublished = {arXiv:2508.14803},
  year         = {2025}
}

@misc{SakaiGoda2026,
  author       = {Sakai, N. and Goda, T.},
  title        = {Space-filling lattice designs for computer experiments},
  howpublished = {arXiv:2602.15390},
  year         = {2026}
}

\end{document}